\documentclass[12pt]{amsart}
\usepackage[T2A]{fontenc}
\usepackage{inputenc}
\usepackage{amsmath}
\usepackage{amsfonts,amssymb}
\usepackage{amsthm}
\usepackage[matrix,arrow,curve]{xy}
\numberwithin{equation}{section}
\theoremstyle{plain}
\newtheorem{theorem}{Theorem}[section]
\newtheorem{lemma}[theorem]{Lemma}
\newtheorem{predl}[theorem]{Proposition}

\theoremstyle{definition}
\newtheorem{definition}[theorem]{Definition}
\newtheorem{remark}[theorem]{Remark}
\newtheorem{example}[theorem]{Example}

\newcommand{\quot }{/\!\!/}

\newcommand{\N}{\mathbb N}
\newcommand{\Z}{\mathbb Z}
\newcommand{\Gm}{\mathbb G_m}

\renewcommand{\P}{\mathbb P}
\renewcommand{\AA}{\mathcal A}
\newcommand{\BB}{\mathcal B}
\newcommand{\CC}{\mathcal C}

\newcommand{\ZZ}{\mathcal Z}
\newcommand{\D}{\mathcal D}
\newcommand{\Dp}{\mathcal D^{\mathrm{\perf}}}
\newcommand{\EE}{\mathcal E}
\newcommand{\FF}{\mathcal F}

\newcommand{\LL}{\mathcal L}
\newcommand{\NN}{\mathcal N}
\newcommand{\TT}{\mathsf T}
\renewcommand{\O}{\mathcal O}

\renewcommand{\k}{\mathsf k}
\newcommand{\Mod}{\mathrm{{-}Mod}}
\newcommand{\mmod}{\mathrm{{-}mod}}

\newcommand{\coh}{\mathrm{coh}}
\newcommand{\qcoh}{\mathrm{qcoh}}
\newcommand{\perf}{\mathrm{perf}}
\newcommand{\Id}{\mathrm{Id}}
\newcommand{\Kern}{\mathrm{Kern}}

\newcommand{\ra}{\mathbin{\rightarrow}}

\newcommand{\xra}{\xrightarrow}
\renewcommand{\le}{\leqslant}
\renewcommand{\ge}{\geqslant}

\newcommand{\bul}{\bullet}

\def\a{\alpha}
\def\b{\beta}

\newcommand{\e}{\varepsilon}
\newcommand{\s}{\sigma}

\DeclareMathOperator{\Hom}{\textup{Hom}}

\DeclareMathOperator{\Tor}{\mathrm{Tor}} \DeclareMathOperator{\Spec}{\mathrm{Spec}}
 \DeclareMathOperator{\Ob}{\mathrm{Ob}}
 
\DeclareMathOperator{\repr}{\mathrm{rep}} 

\begin{document}

\title{Descent theory for semiorthogonal decompositions}
\author{└lexey ELAGIN}
\address{Institute for Information Transmission Problems (Kharkevich Institute)\\
National Research University Higher School of Economics}
\email{alexelagin@rambler.ru}
\thanks{This work was partially supported by 
RFBR {(grants nos. 09-01-00242-a, 10-01-93110-═╓═╚╦-a, 10-01-93113-═╓═╚╦-a, and 11-01-92613-╩╬-р)}, Russian Presidential grant ═╪-5139.2012.1, the Academic Fund of the National Research University ``Higher School of
Economics'' (grant no. 10-09-0015), and Dmitry Zimin's Foundation ``Dynasty''. The author was also partially supported by
AG Laboratory NRU-HSE, RF government 
grant, ag. 11.G34.31.0023}
\date{}

\begin{abstract}
\sloppy
In this paper a method of constructing a semiorthogonal decomposition of the derived category of 
$G$-equivariant sheaves on a variety $X$ is described, provided that the derived category of sheaves on $X$ admits a semiorthogonal decomposition, whose components are preserved by the action of the group $G$ on $X$. 
Using this method, semiorthogonal decompositions of equivariant derived categories were obtained for projective bundles and for blow-ups with a smooth center, and also for varieties with a full exceptional collection, preserved by the action of the group. 
As a main technical instrument, descent theory for derived categories is used.

Reference: 12 items.
\end{abstract}

\maketitle

\markright{Descent theory for semiorthogonal decompositions}

\sloppy
\section{Introduction}

This paper is devoted to studying the derived category of coherent sheaves on an algebraic variety, acted by an algebraic group. Its main aim is to construct  a semiorthogonal decomposition of this category.

We present a construction producing a semiorthogonal decomposition of the equivariant derived category starting from a semiorthogonal decomposition of the standard derived category. This problem can be naturally generalized to the following setup:
what is the relation between the derived categories of the base variety and the covering variety? In the case of the derived category of  $G$-equivariant sheaves on a variety 
$X$, acted   by a group~$G$, the role of the covering variety is played by $X$, and the role of the base is played by the stack $X\quot G$ (the quotient stack of $X$ over the action of the group $G$). This is why one should work in the category of stacks, not only algebraic varieties or schemes.

For a morphism of stacks $p\colon X\ra S$, there is a standard way to 
describe the category of sheaves on $S$ in terms of the category of sheaves on $X$. Namely, if the morphism $p$ is faithfully flat, then giving a sheaf on~$S$ is equivalent to giving a sheaf $F$ on $X$ with the following gluing data: an isomorphism $p_1^*F\ra p_2^*F$ on
$X\times_SX$, satisfying the cocycle condition. 
The similar result holds for an object of the derived category of sheaves on $S$ under the following splitting assumption: the natural morphism $\O_S\ra Rp_*\O_X$ is an embedding of the sheaf $\O_S$ as a direct summand. In other words, under this assumption the derived category of sheaves on the base $S$ is equivalent to a certain descent category, associated with the morphism $p\colon X\ra S$. 
Thus, the language of descent data can be used for understanding the relation between semiorthogonal decompositions of the derived categories of sheaves on the base and on the covering variety. For the morphism $X\to X\quot G$, the  splitting assumption 
means linear reductivity of the group $G$, i.e. that the category of linear representations of $G$ is semisimple.

The descent category, associated with the morphism $p\colon X\ra S$, has two equivalent definitions. The first one: it is the classic descent category $\D(X)/p$ formed by pairs, which consist of an object $F$ in the unbounded derived category $\D(X)$ of quasi-coherent sheaves on $X$ and of an isomorphism $p_1^*F\ra p_2^*F$, obeying the cocycle condition. The second definition is: the descent category  is the category $\D(X)_{\TT_p}$ of comodules over the comonad
$\TT_p=(T_p,\e,\delta)$ on the category $\D(X)$, where $\TT_p$ is the comonad  associated with the adjoint pair of functors $p^*$ and $p_*$. The language of comodules over  a comonad is sometimes more convenient: in terms of comodules over a comonad we formulate Theorem~\ref{th_sodt}, where a semiorthogonal decomposition of the descent category is constructed, provided that some semiorthogonal decomposition of the initial category is compatible in a certain sense with the functor $T_p$.

From Theorem~\ref{th_sodt} we deduce  our main results on the relation between derived categories of the base variety and the covering variety:  Theorems~\ref{th_sodschemesbig}, \ref{th_sodschemes} and \ref{th_sodschemescoh} for a covering of schemes and Theorems~\ref{th_sodforequivbig}--\ref{th_sodforequivcoh} for an  equivariant derived category. Assume that a semiorthogonal decomposition of the category $\Dp(X)$ of perfect complexes on  $X$ is preserved by an action of a linearly reductive group $G$ on $X$. Then Theorem~\ref{th_sodforequiv} produces a semiorthogonal decomposition of $\D^{\perf,G}(X)$, the category of 
$G$-equivariant perfect complexes on $X$. The components of this decomposition can be described in terms of descent theory. 

We provide applications of Theorem~\ref{th_sodforequiv} in the following situations: action of a group on the projectivization of an equivariant vector bundle and on the blow-up of a smooth subvariety. In both cases Theorem~\ref{th_sodforequiv} is applied to the semiorthogonal decompositions of the derived category of sheaves on the above varieties, constructed by D.\,Orlov. In these cases we describe explicitly the components of the resulting decompositions as certain equivariant derived categories.

Another important application Theorem~\ref{th_sodforequiv} has in  the case of an action 
of a linearly reductive group which preserves a full exceptional collection on the variety. This is the case of the simplest  semiorthogonal decomposition, invariant under the action of the group. In this case we also can describe explicitly 
(Theorem~\ref{th_sodforexcoll}) the components of the decomposition from Theorem~\ref{th_sodforequiv}. They are equivalent to the derived categories of representations of the group, twisted into certain cocycles. Earlier the similar result was obtained in~\cite{El_C1} under the assumption that the exceptional collection is formed by sheaves. Using descent theory allows to drop off this assumption.

The author is grateful to 
D.\,O.\,Orlov for his constant attention to this work and for useful discussions,  and to A.\,G.\,Kuznetsov for valuable remarks, refereeing this text and for his kind support.

\section{Preliminaries}
\label{section_pre}
All schemes in this paper are supposed to be quasi-projective over an arbitrary field $\k$.
By an algebraic group we understand an affine group scheme of finite type over $\k$. The structure morphism in a group $G$ we denote by $\mu\colon G\times G\ra G$, and the action of a group $G$ on a scheme $X$  by $a\colon G\times X\ra X$.
The projections of products $G\times X$, $G\times G\times X$,  etc onto factors we denote by $p_1$, $p_2$, $p_{12}$, etc.

We will consider three versions of derived category of sheaves on a scheme $X$: the unbounded derived category of quasi-coherent sheaves $\D(X)=\D(\qcoh(X))$, the bounded derived category of coherent sheaves $\D^b(X)=\D^b(\coh(X))$ and the category of perfect complexes $\D^{\perf}(X)$. The latter is by definition a full subcategory in $\D(X)$ formed by complexes that are locally quasi-isomorphic to a bounded complex of vector bundles. One has $\D^{\perf}(X)\subset \D^b(X)$ and for a smooth scheme one has $\D^{\perf}(X)=\D^b(X)$, see~\cite{BvdB}. Let $f\colon X\ra Y$ be a morphism of schemes, then by Spaltenstein the derived  pullback and pushforward functors between categories $\D(X)$ and $\D(Y)$ are well-defined (see~\cite{Sp}). These functors are adjoint, we denote them by $f^*$ and $f_*$ by abuse of notation. Note that $f^*$ sends $\D^{\perf}(Y)$ to
$\D^{\perf}(X)$, and for  a morphism~$f$ of finite $\mathrm{Tor}$-dimension (in particular,
for flat $f$) it sends $\D^b(Y)$ to $\D^b(X)$. If $f$ is proper,  then $f_*$ sends $\D^b(X)$ to $\D^b(Y)$.

\begin{definition}
An \emph{equivariant sheaf} on a scheme $X$ acted by a group~$G$ is by definition a pair of a sheaf $F$ and of an isomorphism $\theta\colon p_2^*F\ra a^*F$ of sheaves on $G\times X$ that satisfy on $G\times G\times X$ the cocycle condition: $(1\times a)^*\theta\circ p_{23}^*\theta=(\mu\times 1)^*\theta$.
A \emph{morphism} of equivariant sheaves from  $(F_1,\theta_1)$ to $(F_2,\theta_2)$ is by definition a morphism $f\colon F_1\ra F_2$ compatible with~$\theta$: that is,
$\theta_2\circ p_2^*f=a^*f\circ \theta_1$.
\end{definition}
Coherent (quasi-coherent) $G$-equivariant sheaves on a scheme $X$ form an abelian category $\coh^G(X)$ (resp. $\qcoh^G(X)$). We denote by $\D^G(X)=\D(\qcoh^G(X))$ the unbounded derived category of $G$-equivariant quasi-coherent sheaves on $X$, we denote by $\D^b(\coh^G(X))$ the bounded derived category of $G$-equivariant coherent sheaves on $X$. We denote by $\D^{\perf,G}(X)$ the category of $G$-equivariant perfect complexes, it is a full subcategory in $\D^G(X)$ formed by complexes that lie in $\D^{\perf}(X)$ after forgetting of the group action.

One can view  $G$-equivariant sheaves on a scheme $X$ as sheaves on the stack $X\quot G$, the quotient stack of $X$ over the action of $G$. Definition of $X\quot G$ can be found in~\cite{La}, under our hypotheses $X\quot G$ is an algebraic stack of finite type over $\k$.
There is a canonical flat morphism of stacks $p\colon X\ra X\quot G$, the pullback under $p$ is forgetting of the equivariant structure.
This approach allows one to describe the derived category of equivariant sheaves on a scheme in terms of the derived category of sheaves.
There are two ways of such description (which are essentially equivalent): using cosimplicial categories and using comonads. We present these ways below.

Let  $\Delta$ denote a category whose objects are sets $[1,\ldots,n], n\in\N$, and
whose  morphisms are non-decreasing maps between them.

By definition, a cosimplicial object in a category $\CC$ (for example, a cosimplicial set, a cosimplicial scheme, etc) is a functor from $\Delta$
to~$\CC$. If we take $\CC$ to be the 2-category of categories $\mathrm{Cats}$, we get a definition of cosimplicial category.

\begin{definition}
A \emph{cosimplicial category} is a covariant 2-functor $\Delta \ra \mathrm{Cats}$ into the 2-category of categories. More explicitly,
a cosimplicial category $\CC_{\bul}$ consists of the following data:
\begin{enumerate}
\item a family of categories $\CC_k$, $k=0,1,2,\ldots$, which correspond to objects of $\Delta$ (note that the category $\CC_k$ corresponds to the set $[1,\ldots,k+1]$);
\item a family of functors $P_f^*\colon\CC_m\ra \CC_n$, which correspond to morphisms in $\Delta$, that is, to non-decreasing maps
$f\colon [1,\ldots,m+1]\ra [1,\ldots,n+1]$;
\item a family of isomorphisms of functors $\epsilon_{f,g}\colon P_f^*P_g^*\ra P_{fg}^*$, which correspond to composable pairs of maps $f,g$.
\end{enumerate}
Isomorphisms in (3) should obey the cocycle condition: the diagram
$$\xymatrix{P_f^*P_g^*P_h^* \ar[r]^{\epsilon_{f,g}}\ar[d]^{\epsilon_{g,h}} & P_{fg}^*P_h^*\ar[d]^{\epsilon_{fg,h}}\\ P_f^*P_{gh}^* \ar[r]^{\epsilon_{f,gh}} & P_{fgh}^*}$$
is commutative for any composable triple of maps $f,g,h$.
\end{definition}

For us, the main example of a cosimplicial category is the following.
\begin{example}
\label{example_standard}
Suppose a morphism of schemes (or stacks) $X\ra S$ is given. Then the schemes (or stacks) $X,X\times_S X,X\times_S X\times_S X,\ldots$ and the morphisms
$$p_f\colon \underbrace{X\times_SX\times\ldots\times X}_{n\ \text{times}}\ra \underbrace{X\times_SX\times\ldots\times X}_{m\ \text{times}}$$
between them given by the rule
$$p_f(x_1,\ldots,x_n)=(x_{f(1)},\ldots,x_{f(m)})$$
for $f\in \Hom_{\Delta}([1,\ldots,m],[1,\ldots,n])$,
form a simplicial scheme (or stack). And the categories of sheaves (abelian categories of coherent, quasi-coherent sheaves or corresponding derived categories)
on these schemes (stacks) and pullback functors between them form a cosimplicial category.
\end{example}
We will denote such cosimplicial categories by
$$[\D(X),\D(X\times_SX),\D(X\times_SX\times_SX),\ldots,p_f^*].$$
It is natural to use for functors $P_f^*$ the notation, reminding of pullback functors for the categories of sheaves.
For instance, we denote by $P_{13}^*$ the functor $P_{f}^*\colon \CC_1\ra \CC_2$ for the map $f\colon [1,2]\ra [1,2,3]$ such that $f(1)=1$, $f(2)=3$.

For any cosimplicial category $\CC_{\bul}=[\CC_{0}, \CC_1,
\ldots, P_{\bul}^*]$, the descent category $\Kern(\CC_{\bul})$ is defined, see~\cite[19.3]{KSh}.

\begin{definition}[{Classic descent category}]
\label{def_dxbar} An \emph{object} of $\Kern(\CC_{\bul})$ is a pair 
$(F,\theta)$, where $F\in \Ob  \CC_0$ and  $\theta$ is an isomorphism $P_1^*F\ra
P_2^*F$ satisfying the cocycle condition: the following diagram is commutative $$\xymatrix{P_{12}^*P_1^*F \ar[rd]_{P_{12}^*\theta}&P_{13}^*P_1^*F \ar[r]^{P_{13}^*\theta} \ar@{-}[l]_{\sim} &
P_{13}^*P_2^*F \ar@{-}[rd]^{\sim}&\\
&P_{12}^*P_2^*F \ar@{-}[r]_{\sim} & P_{23}^*P_1^*F \ar[r]_{P_{23}^*\theta}&  P_{23}^*P_2^*F.}$$
In this diagram lines with $\sim$ denote functorial isomorphisms from the definition of a cosimplicial category.
A \emph{morphism} in $\Kern(\CC_{\bul})$ from $(F_1,\theta_1)$ to
$(F_2,\theta_2)$ is a morphism $f\in \Hom_{\CC_0}(F_1,F_2)$ such that 
$P_2^*f\circ \theta_1=\theta_2\circ P_1^*f$.
\end{definition}

For the canonical morphism of stacks $X\ra X\quot G$, the simplicial stack from Example~\ref{example_standard} is actually a simplicial scheme and has the form 
\begin{equation}
\label{equation_X/G}
[X_0,X_1,X_2,\ldots,p_{\bul}]=[X,G\times X,G\times G\times X,\ldots,p_{\bul}],
\end{equation}
where the morphisms $p_{\bul}$ are defined as follows. For a non-decreasing map $f\colon [1,\ldots,m]\ra [1,\ldots,n]$ the morphism of schemes 
$$p_f\colon \underbrace{G\times \ldots\times G\times X}_{n\ \text{times}}\ra
\underbrace{G\times \ldots\times G\times X}_{m\ \text{times}}$$
is given by the rule
$$(g_n,\ldots,g_2,x_1)\mapsto (g_{f(m)}\cdot\ldots\cdot g_{f(m-1)+1},\ldots,
g_{f(2)}\cdot\ldots\cdot g_{f(1)+1},g_{f(1)}\cdot\ldots\cdot g_2x_1).$$
For small $n=m\pm 1$ and strictly increasing $f$ the morphisms $p_{\bul}$
have the form 
$$\xymatrix{
G\times G\times X \ar@<-10mm>[rr]^-{p_{23}} \ar[rr]^-{\mu\times 1} \ar@<10mm>[rr]^-{1\times a} &&
G\times X \ar@<-5mm>[rr]^-{p_2} \ar@<5mm>[rr]^-{a}\ar@<5mm>[ll]_-{e\times 1}
\ar@<-5mm>[ll]_-{p_1\times e\times p_2}&& X \ar[ll]_-{e\times 1}}.$$

By the definition, for a cosimplicial category $$[\qcoh(X),\qcoh(G\times X),\qcoh(G\times G\times X),\ldots,p^*],$$
formed by abelian categories of quasi-coherent sheaves on~(\ref{equation_X/G}),
the descent category $\Kern$ is exactly the category of $G$-equivariant quasi-coherent sheaves on $X$.
We will see below that for a linearly reductive group $G$
the descent category 
\begin{equation}
\label{equation_DX/G}
\D(X)^G=\Kern([\D(X),\D(G\times X),\D(G\times G\times X),\ldots,p^*])
\end{equation}
is equivalent to the derived category of equivariant sheaves $\D^G(X)$.
Informally, taking descent category commutes with taking derived category.

\bigskip
Now we recall necessary definitions and facts concerning comonads and comodules over a comonad. More information can be found in the books of Barr-Wells~\cite[ch. 3]{TTT} and MacLane~\cite[ch.
6]{ML}.

Let $\CC$ be a category.
\begin{definition}
\label{def_comonad}
A \emph{comonad} $\TT=(T,\e,\delta)$ on the category $\CC$ consists of a functor
$T\colon \CC\ra \CC$ and of two natural transformations of functors 
$\e\colon T\ra \Id_{\CC}$ and $\delta\colon T\ra T^2=TT$ such that the diagrams
$$\xymatrix{
T \ar[r]^{\delta} \ar@{=}[rd] \ar[d]^{\delta} & T^2 \ar[d]^{T\e} \\
T^2 \ar[r]^{\e T} & T, }\qquad
\xymatrix{T \ar[r]^{\delta} \ar[d]^{\delta} & T^2 \ar[d]^{T\delta}\\
T^2 \ar[r]^{\delta T} & T^3 }
$$
commute.
\end{definition}

The main (and, essentially, the only) example of a comonad is the following
\begin{example}
\label{mainexample} Consider a pair of adjoint functors 
$P^*
\colon\BB\ra \CC$ (left) and $P_*\colon \CC\ra \BB$ (right).
Let $\eta\colon \Id_{\BB}\ra P_*P^*$ and $\e\colon P^*P_*\ra \Id_{\CC}$ 
be the natural adjunction morphisms. Define a triple 
$(T,\e,\delta)$: take $T=P^*P_*$, take 
$\e\colon P^*P_*\ra \Id_{\CC}$ and $\delta=P^*\eta P_* \colon P^*P_*\ra
P^*P_*P^*P_*$.
Then $\TT=(T,\e,\delta)$ is a comonad on the category~$\CC$.
\end{example}

In fact, any comonad can be obtained from an adjoint pair in this way. It follows from the construction due to Eilenberg and Moore.

\begin{definition}
\label{def_comodule}
Suppose $\TT=(T,\e,\delta)$ is a comonad on the category~$\CC$.
A \emph{comodule} over~$\TT$ (or a \emph{$\TT$-coalgebra}) is a pair 
$(F,h)$, where  $F\in \Ob  \CC$ and $h\colon F\ra TF$ is a morphism satisfying two conditions: the composition
$$F\xra{h} TF \xra{\e F} F$$
is the identity, the diagram 
$$\xymatrix{
F \ar[r]^h \ar[d]^h  &  TF \ar[d]^{Th} \\
TF \ar[r]^{\delta F} & T^2F }$$ is commutative.
A \emph{morphism} between comodules
$(F_1,h_1)$ and $(F_2,h_2)$ is a morphism 
$f\colon F_1\ra F_2$ in the category $\CC$ such that $h_2\circ f=Tf\circ h_1$.
\end{definition}

Comodules over a fixed comonad $\TT$ on $\CC$ form a category, which is denoted by $\CC_{\TT}$. Define $P^*\colon \CC_{\TT}\to \CC$ to be the forgetful functor: 
$(F,h)\mapsto F$. Define a functor 
$P_*\colon \CC\to \CC_{\TT}$ by the rule $F\mapsto (TF,\delta F)$. 
It can be shown (see~\cite[section 3.2]{TTT}) that the functors $P^*$ and $P_*$ 
are adjoint and that the comonad~$\TT$ is isomorphic to the comonad, constructed from the adjoint pair $P^*,P_*$.

Also we will need more general definition:
\begin{definition}
\label{def_dxt'}
Let $\TT$ be a comonad on $\CC$ and  $\CC'\subset\CC$ be a subcategory. 
Define
$\CC'_{\TT}$ to be the full subcategory in the category of comodules over $\TT$
formed by comodules $(F,h)$ such that $F\in \CC'$.
\end{definition}

Consider the commutative square in the category  $\Delta$
$$\xymatrix{[1,\ldots,m+n-r]&[1,\ldots,n] \ar[l]_-{f'}\\
[1,\ldots,m] \ar[u]_{g'} & [1,\ldots,r], \ar[l]_{f}
\ar[u]_g}.$$
\begin{definition}
If maps $f$ and $f'$ are injective and $[1,\ldots,m+n-r]=\mathrm{Im}\ f'\cup \mathrm{Im}\ g'$, then this square is called \emph{exact Cartesian}.
\end{definition}

\begin{definition}[see~\cite{El_C4}, section 2]
\label{def_basechangecat}
Let $[\CC_0,\CC_1,\ldots,P^*_{\bul}]$  be a cosimplicial category.
Suppose that 
\begin{enumerate}
\item the  functors $P_f^*$ have right 
adjoint functors, denote them by $P_{f*}$;
\item for any exact Cartesian square in $\Delta$
the natural base change morphism 
$$P_f^*P_{g*}\ra P_{g'*}P_{f'}^*$$
is a functor isomorphism.
\end{enumerate}
Then $[\CC_0,\CC_1,\ldots,P^*_{\bul}]$  is called
a \emph{cosimplicial category with base change}.
\end{definition}
Condition (2) above is 
an axiomatization of the flat base change theorem.

Cosimplicial categories formed by categories of sheaves (either abelian or derived) from Example~\ref{example_standard} are cosimplicial categories with base change.

For any cosimplicial category with base change $\CC_{\bul}=[\CC_0,\CC_1,\ldots,P_{\bul}^*]$  one can construct a comonad 
on the category $\CC_0$. Let $P^*\colon \Kern(\CC_{\bul})\ra\CC_0$ be the forgetful functor, which assigns the object $F$ to a pair $(F,\theta)$.
By~\cite[prop. 2.9]{El_C4} the functor $P^*$ has a right adjoint functor $P_*$.
\begin{definition}
\label{def_simpl>comonad}
Define $\TT_{\CC_{\bul}}=(P^*P_*,\e,\delta)$ to be the comonad on $\CC_0$
associated with the adjoint pair $(P^*,P_*)$.
\end{definition}
For cosimplicial categories from Example~\ref{example_standard}, 
this comonad coincides with the comonad associated with the adjoint pair of pullback and pushforward functors between the categories of sheaves on $X$ and on~$S$.

\begin{definition}[comonad descent category]
For a cosimplicial category with base change $\CC_{\bul}$, let 
$\TT_{\CC_{\bul}}$ be the comonad on
$\CC_0$, associated with~$\CC_{\bul}$ as above.
Define $\CC_{\bul\TT}$ to be the category of comodules over $\TT_{\CC_{\bul}}$.
\end{definition}
\begin{predl}[{see \cite{El_C4}, prop. 4.2}]
\label{prop_simpl>comonad}
For any cosimplicial category with base change $\CC_{\bul}$ 
the descent categories $\Kern(\CC_{\bul})$ and $\CC_{\bul\TT}$ are equivalent.
\end{predl}

Clearly, if $\TT=(T,\e,\delta)$ is a comonad on an abelian category $\CC$ and the functor $T$ is exact then the category of comodules $\CC_{\TT}$ is also abelian.
But it is not clear why the category~$\CC_{\TT}$ should be triangulated if $\CC$ 
is a triangulated category, $\TT=(T,\e,\delta)$ is a comonad on $\CC$ and the functor $T$ is exact.                                                                   There is a good candidate of triangulated structure on $\CC_{\TT}$:
\begin{definition}
\label{def_quasitriang}
Suppose $\CC$ is a triangulated category and $\TT=(T,\e,\delta)$ is a comonad on~$\CC$. Define shift functor on $\CC_{\TT}$ by formulas
$(F,h)[1]=(F[1],h[1])$, $f[1]=f[1]$. Say that a triangle $(F',h')\ra (F,h)\ra (F'',h'')\ra (F',h')[1]$ in $\CC_{\TT}$ is \emph{distinguished} if the triangle 
$F'\ra F\ra F''\ra F'[1]$ is distinguished in $\CC$.
\end{definition}
Since taking cones is not functorial, we cannot check without additional assumptions that any morphism in
$\CC_{\TT}$ fits into a distinguished triangle.
But later we will see that in some interesting situations the above definition does introduce a triangulated structure on $\CC_{\TT}$.

\medskip
Suppose $p\colon X\ra S$ is a morphism of schemes and 
$$\CC_{\bul}=[\D(X),\D(X\times_SX),\ldots,p_{\bul}^*]$$
is the cosimplicial category formed by derived categories of quasi-coherent sheaves from Example~\ref{example_standard}. Denote by $\TT_p$ the corresponding comonad on $\D(X)$. There is a canonical functor 
$$\Phi\colon \D(S)\ra \D(X)_{\TT_p},$$
called a \emph{comparison functor}, which sends a complex $H$ to the pair  $(p^*H,h)$, where $h\colon p^*H\to p^*p_*p^*H$ is the canonical adjunction morphism.

\medskip
Now we state the main theorems about equivalence between derived descent category and derived category of the base scheme.

\begin{theorem}[{see~\cite[th.~7.3]{El_C4}}]
\label{th_descentfor}
Suppose that for a morphism $p\colon X\ra S$ of quasi-projective schemes over a field the canonical map $\O_S\to Rp_*\O_X$ is a split embedding. Then
the comparison functors $$\Phi\colon \D(S)\ra \D(X)_{\TT_p},\qquad
\Phi\colon \D^{\perf}(S)\ra \D^{\perf}(X)_{\TT_p}, \qquad
\Phi\colon \D^b(S)\ra \D^b(X)_{\TT_p}$$
are equivalences.
\end{theorem}

To formulate another result we need to recall
\begin{definition}[{see~\cite[def.~1.4]{MF}}]
A group scheme $G$ over a field~$\k$ is \emph{linearly reductive} if
the category of finite dimensional representations of $G$ over $\k$ is semisimple.
\end{definition}

\begin{theorem}[{see~\cite[th.~9.6]{El_C4}}]
\label{th_descentfor2}
Suppose that a linearly reductive group scheme $G$ of finite type over a field $\k$ acts on a quasi-projective scheme $X$ over $\k$. Then the comparison functors $$\Phi\colon \D^G(X)\ra \D(X)_{\TT_G},\quad
\Phi\colon \D^{\perf,G}(X)\ra \D^{\perf}(X)_{\TT_G},\quad
\Phi\colon \D^b(\coh^G(X))\ra \D^b(X)_{\TT_G}$$
are equivalences (here $\TT_G$ is a comonad on $\D(X)$, associated with
(\ref{equation_DX/G})).
\end{theorem}

The equivalence between $\D(X)_{\TT_p}$ and the triangulated category
$\D(S)$ obtained above allows to carry triangulated structure from $\D(S)$ 
to $\D(X)_{\TT_p}$, the same for $\D^{\perf}(X)_{\TT_p}$ and $\D^b(X)_{\TT_p}$. 
The resulting triangulated structure on $\D(X)_{\TT_p}$ coincides with the one
from Definition~\ref{def_quasitriang}, see~\cite[prop. 3.13]{El_C4}.

The operation of taking  descent category is functorial. 
To say it formally, a notion of ``morphism'' for cosimplicial categories and for categories with a comonad is needed.

\begin{definition}
\label{def_functorcosimpl}
A \emph{functor between cosimplicial categories} $\CC_{\bul}^{(1)}$ and $\CC_{\bul}^{(2)}$ consists of a family of functors     $$\Psi_k\colon
\CC_k^{(1)}\ra \CC_k^{(2)},\qquad k=0,1,\ldots$$
and a family of isomorphisms of functors $$\b_f\colon \Psi_n\circ P_f^{(1)*}\xra{\sim} P_f^{(2)*}\circ \Psi_m,$$ parametrized by  maps 
$f\colon [1,\ldots,m+1]\ra[1,\ldots,n+1]$ in $\Delta$. The isomorphisms should be  compatible with the composition of maps in~$\Delta$.
\end{definition}

\begin{definition}
Let $\CC_i, i=1,2$, be two categories and $\TT_i=(T_i,\e_i,\delta_i)$ be two comonads on them.  We say that a functor $\Psi\colon \CC_1\ra\CC_2$ is \emph{compatible} with $\TT_1$ and
$\TT_2$ if there exists an isomorphisms of functors $\b\colon\Psi T_1\ra T_2\Psi$ such that the diagrams 
$$\xymatrix{
\Psi T_1\ar[rr]^{\b} \ar[rd]_{\Psi\e_{1}}&& T_2\Psi \ar[ld]^{\e_2\Psi} \\ &\Psi,&
}\qquad\xymatrix{ \Psi T_1 \ar[rr]^{\b} \ar[d]_{\Psi\delta_1}&& T_2\Psi
\ar[d]^{\delta_2\Psi}\\ \Psi T_1^2 \ar[r]^{\b T_1} & T_2\Psi T_1 \ar[r]^{T_2\b} &
T_2^2\Psi }$$ are commutative.
\end{definition}
More formally, one should think of the isomorphism $\b$ as of a part of the data and define morphisms in the 2-category of categories with a comonad as pairs $(\Psi,\b)$. But we will not make this difference.

The following fact can be easily checked.
\begin{lemma}
\label{lemma_psipsi} Let $(\Psi_k)$ be a functor between cosimplicial categories with base change $\CC_{\bul}^{(1)}$ and 
$\CC_{\bul}^{(2)}$. Let $\TT_1$ and $\TT_2$ be the comonads on $\CC_0^{(1)}$ and $\CC_0^{(2)}$, defined in Definition~\ref{def_simpl>comonad}. Then the functor $\Psi_0\colon
\CC^{(1)}_0\ra\CC^{(2)}_0$ is compatible with comonads $\TT_1$ and
$\TT_2$.
\end{lemma}

The following lemma will be used later for description of components of semiorthogonal decompositions.
\begin{lemma}
\label{lemma_fffunctor}
Let $\TT_1$ and $\TT_2$ be comonads on categories $\CC_1$ and $\CC_2$ respectively, let $\CC'_1\subset\CC_1$ and $\CC'_2\subset \CC_2$ be subcategories.
Suppose that the functor $\Psi\colon \CC_1\ra \CC_2$ sends $\CC_1'$ to $\CC_2'$.
If $\Psi$ is compatible with the comonads  $\TT_1$ and $\TT_2$, then $\Psi$ induces a functor $\Psi_{\TT}\colon\CC'_{1\TT_1}\ra\CC'_{2\TT_2}$. Moreover, if $\Psi$ 
is fully faithful and is an equivalence between $\CC_1'$ and some full subcategory $\AA'$ in $\CC_2'$, then $\Psi_{\TT}$ is an equivalence 
$\CC_{1\TT_1}'\ra\AA_{\TT_2}'\subset \CC_{2\TT_2}'$.
\end{lemma}
\begin{proof}
Evident.
\end{proof}

\section{Semiorthogonal decompositions for categories of comodules}
\label{section_sodt}

Let $\TT$ be a comonad on a triangulated category $\CC$.
In this section we show that a semiorthogonal decomposition of the category $\CC$ induces a semiorthogonal decomposition of the category of comodules $\CC_{\TT}$ under the following assumptions: the category $\CC_{\TT}$ is triangulated and the initial semiorthogonal decomposition is compatible with~$\TT$.

\begin{definition}
\label{def_compatible}
We say that a functor $T\colon \CC\ra\CC$ is
\emph{upper triangular} with respect to a semiorthogonal decomposition $\CC=
\langle\AA_1,\ldots,\AA_n\rangle$ if
\begin{equation}
T\AA_k\subset
\langle\AA_1,\ldots,\AA_k\rangle\quad\text{for all} \quad k, 1\le
k\le n.
\end{equation}
\end{definition}

\begin{theorem}
\label{th_sodt}Let $\TT=(T,\e,\delta)$ be a comonad on a triangulated category $\CC$, let $\CC'\subset\CC$ be a triangulated subcategory.
Suppose that the functor $T$ is exact and that the category $\CC_{\TT}$ is triangulated in the sense of Definition~\ref{def_quasitriang}.
Suppose that the functor $T$ is upper triangular with respect to a semiorthogonal decomposition
$\CC=\langle\AA_1,\ldots,\AA_n\rangle$. Suppose also that this decomposition induces a semiorthogonal decomposition $\CC'=\langle\AA'_1,\ldots,\AA'_n\rangle$,
where $\AA'_k=\AA_k\cap\CC'$. Then the category $\CC'_{\TT}$ is triangulated  and admits a semiorthogonal decomposition $\CC'_{\TT}=\langle\AA'_{1\TT},\ldots,\AA'_{n\TT}\rangle$.
\end{theorem}

\begin{proof}
It follows directly from the definitions that the categories $\CC'_{\TT}$ and $\AA'_{i\TT}$ are triangulated.

The proof of the theorem is by induction in $n$, first we treat the case $n=2$.

It is evident that the subcategories $\AA'_{1\TT}$ and $\AA'_{2\TT}$ are semiorthogonal.
We need to check that any object $(F,h)$ in $\CC'_{\TT}$ fits into a distinguished triangle
\begin{equation}
\label{FFFh}
(F_2,h_2)\ra (F,h)\ra (F_1,h_1)\ra (F_2,h_2)[1],
\end{equation}
where $F_i\in \AA'_i$ and the morphisms lie in $\CC'_{\TT}$. Since
$\CC'=\langle\AA'_1,\AA'_2\rangle$, there exists a distinguished triangle in the category $\CC'$
\begin{equation}
\label{FFF}
F_2\ra F\ra F_1\ra F_2[1]
\end{equation}
with $F_i\in\AA'_i$. Applying the exact functor $T$  to it, we get a distinguished triangle
$$TF_2\ra TF\ra TF_1\ra TF_2[1].$$
Since $T$ is upper triangular, one has $\Hom(F_2,TF_1)=0$, therefore the morphism $h\colon F\ra TF$ extends to a morphism of triangles
\begin{equation}
\label{FFFFFF}
\xymatrix{
F_2\ar[r] \ar[d]^{h_2} & F\ar[r] \ar[d]^{h} & F_1\ar[r] \ar[d]^{h_1} &
F_2[1] \ar[d]^{h_2[1]} \\
TF_2\ar[r]  & TF\ar[r]  & TF_1\ar[r]  & TF_2[1].
}
\end{equation}
Let us show that the pairs $(F_i,h_i)$ are comodules over $\TT$.
Indeed, consider the compositions of~(\ref{FFFFFF}) with
$$
\xymatrix{
TF_2\ar[r] \ar[d]^{Th_2} & TF\ar[r] \ar[d]^{Th} & TF_1\ar[r] \ar[d]^{Th_1} &
TF_2[1] \ar[d]^{Th_2[1]} \\
TTF_2\ar[r]  & TTF\ar[r]  & TTF_1\ar[r]  & TTF_2[1]
}
$$ and
$$
\xymatrix{
TF_2\ar[r] \ar[d]^{\delta F_2} & TF\ar[r] \ar[d]^{\delta F} & TF_1\ar[r] \ar[d]^{\delta F_1} &
TF_2[1] \ar[d]^{\delta F_2[1]} \\
TTF_2\ar[r]  & TTF\ar[r]  & TTF_1\ar[r]  & TTF_2[1].
}
$$
We obtain the two morphisms of triangles
$$\xymatrix{
F_2\ar[rr] \ar@<-1mm>[d]_{Th_2\circ h_2}\ar@<1mm>[d]^{\delta F_2\circ h_2} && F\ar[rr] \ar[d]_{Th\circ h}^{\delta F\circ h} &&
F_1\ar[rr] \ar@<-1mm>[d]_{Th_1\circ h_1}\ar@<1mm>[d]^{\delta F_1\circ h_1} && F_2[1] \ \ar@<-1mm>[d]_{Th_2\circ h_2[1]}\ar@<1mm>[d]^{\delta F_2\circ h_2[1]}\\
TTF_2 \ar[rr]  && TTF\ar[rr]  && TTF_1\ar[rr]  && TTF_2[1],}$$
they coincide in the middle term.
Since $\Hom(F_2,TTF_1)=0$, a morphism  of triangles, extending a given morphism of middle terms, is unique. Hence, $\delta F_i\circ h_i=Th_i\circ h_i$.
It is checked similarly that $\e F_i\circ h_i=\Id_{F_i}$. The diagram~(\ref{FFFFFF}) shows that the demanded triangle~(\ref{FFFh}) is constructed.

Now suppose that the theorem is proved for $n\ge 2$, consider the case $n+1$. The semiorthogonal decompositions $\AA_0=\langle\AA_1,\ldots,\AA_{n}\rangle$ and $\CC=\langle\AA_0,\AA_{n+1}\rangle$
satisfy assumptions of the theorem, and we obtain
$$\CC'_{\TT}=\langle\AA'_{0\TT},\AA'_{n+1\TT}\rangle=
\langle\langle\AA'_{1\TT},\ldots,\AA'_{n\TT}\rangle,\AA'_{n+1\TT}\rangle=
\langle\AA'_{1\TT},\ldots,\AA'_{n\TT},\AA'_{n+1\TT}\rangle.$$
\end{proof}

\section{Descent for semiorthogonal decompositions: morphism of schemes}
\label{section_descentforschemes}

Let $p\colon X\ra S$ be a flat morphism of quasi-projective schemes such that $\O_S$ is a direct summand in $Rp_*\O_X$. Then (see Theorem~\ref{th_descentfor}) the derived category of sheaves on $S$ is equivalent to a descent category associated with the derived category of sheaves on $X$. This allows to use Theorem~\ref{th_sodt} and thus to construct semiorthogonal decompositions of the derived category of sheaves on $S$.
Here the role of the category $\CC$ from Theorem~\ref{th_sodt} is played by the unbounded derived category $\D(X)$. Unfortunately, the functor
$T=p^*p_*$ on $\D(X)$ usually does not preserve ``small'' subcategories like $\D^{\perf}(X)$ and $\D^b(X)$.
Hence, to use the results of the previous section we need to construct semiorthogonal decompositions of the ``big'' category $\D(X)$.

Denote by $\TT_p=(T_p,\e,\delta)$ the comonad on $\D(X)$ associated with the morphism~$p$.

\begin{theorem}
\label{th_sodschemesbig}
Let $X$ and $S$ be quasi-projective schemes and $p\colon X\ra S$ be a flat morphism such that $\O_S$ is a direct summand in $Rp_*\O_X$.
Suppose a semiorthogonal decomposition $\D(X)=\langle\AA_1,\ldots,\AA_n\rangle$ is given, and the functor $T_p=p^*p_*$ is upper triangular with respect to it. Then the category
$\D(S)$ admits a semiorthogonal decomposition
$\langle\BB_1,\dots,\BB_n\rangle$. Here $\BB_k\subset\D(S)$ denotes a full subcategory that consists of objects~$H$ such that $p^*H\in\AA_k$.
\end{theorem}
\begin{proof}
By Theorem~\ref{th_descentfor}, the category $\D(S)$ is equivalent to
$\D(X)_{\TT_p}$. Now the statement follows from Theorem~\ref{th_sodt} applied to $\CC=\CC'=\D(X)$, from the construction of the equivalence $\D(S)\cong \D(X)_{\TT_p}$, and from Definition~\ref{def_dxt'} of categories $\AA_{i\TT_p}$.
\end{proof}

To prove similar statements for the category of perfect complexes or for the bounded derived category, we need to extend a semiorthogonal decomposition of these categories to a semiorthogonal decomposition of the unbounded derived category. Such extension was done by A.\,Kuznetsov in~\cite{Ku}, we recall the construction.

Let $\AA\subset \D(X)$ be a subcategory.  Define $\AA^{\oplus\infty}\subset \D(X)$
to be the minimal triangulated subcategory in $\D(X)$ containing $\AA$
and closed under arbitrary direct sums.

\begin{lemma}
\label{lemma_extension}
1. Suppose the category $\D^{\perf}(X)$ has a semiorthogonal decomposition
$\langle\AA^{\perf}_1,\ldots,\AA^{\perf}_n\rangle$. Let
 $\AA_i=(\AA^{\perf}_i)^{\oplus\infty}$. Then the categories $\AA_i$ form a semiorthogonal decomposition $\D(X)=\langle \AA_1,\ldots,\AA_n\rangle$,
and $\AA_i^{\perf}=\AA_i\cap \D^{\perf}(X)$.

2. Suppose $\AA^p\subset \D^{\perf}(X)$ is a left admissible triangulated subcategory, and $\D^{\perf}(X)=\langle\AA^{p},\ZZ^{p}\rangle$ is a corresponding semiorthogonal decomposition. Then $\AA=(\AA^p)^{\oplus\infty}$ is a right orthogonal to
$\ZZ^{p}$ in $\D(X)$.

3.  Suppose $\D^b(X)=\langle\AA'_1,\ldots,\AA'_n\rangle$ is a semiorthogonal decomposition into admissible subcategories.  Then the categories $\AA_i=(\AA'_i)^{\oplus\infty}$ form a semiorthogonal decomposition $\D(X)=\langle \AA_1,\ldots,\AA_n\rangle$, and $\AA'_i=\AA_i\cap \D^b(X)$.
\end{lemma}
\begin{proof}
1. Follows from~\cite[prop. 4.2]{Ku}.

2. By 1, there is a semiorthogonal decomposition
$\D(X)=\langle(\AA^p)^{\oplus\infty},(\ZZ^p)^{\oplus\infty}\rangle$.
Therefore,
$$(\AA^p)^{\oplus\infty}=((\ZZ^p)^{\oplus\infty})^{\perp}=
(\ZZ^p)^{\perp}.$$

3.
For any $i$ consider a semiorthogonal decomposition $\D^{b}(X)=\langle\AA'_i,\ZZ'_i\rangle$, take
$$\AA_i^p=\AA'_i\cap \D^{\perf}(X), \quad \ZZ_i^p=\ZZ'_i\cap \D^{\perf}(X),\quad    \AA_i=(\AA_i^p)^{\oplus\infty}.$$
By~\cite[prop. 4.1]{Ku}, we get a decomposition $\D^{\perf}(X)=\langle\AA_i^p,\ZZ_i^p\rangle$.
It follows from~2 that $\AA'_i\subset(\ZZ'_i)^{\perp}\subset (\ZZ^p_i)^{\perp}=
\AA_i$. Hence, $\AA_i=(\AA'_i)^{\oplus\infty}$.
Finally, \cite[prop. 4.1]{Ku} implies that there is a decomposition
$\D^{\perf}(X)=\langle\AA_1^p,\ldots,\AA_n^p\rangle$. Using 1, we obtain a semiorthogonal decomposition $\D(X)=\langle\AA_1,\ldots,\AA_n\rangle$.
The last statement follows from~\cite[lemma 3.2]{Ku}.
\end{proof}

\begin{theorem}
\label{th_sodschemes}
Let $X, S$ and $p$ be as in Theorem~\ref{th_sodschemesbig}, let
$p_1$ and $p_2$ be two projections $X\times_SX\to X$.
Suppose that the category of perfect complexes  $\D^{\perf}(X)$ has
a semiorthogonal decomposition
$\langle\AA^{\perf}_1,\ldots,\AA^{\perf}_n\rangle$. Suppose
$\Hom(p_1^*F_j,p_2^*F_i)=0$ for all $1\le i<j\le n$ and objects
$F_i\in \AA^{\perf}_i,
    F_j\in\AA^{\perf}_j$.
Then $\D^{\perf}(S)$ has a semiorthogonal decomposition $\langle\BB^{\perf}_1,\ldots,\BB^{\perf}_n\rangle$, where
$\BB^{\perf}_k\subset\D^{\perf}(S)$ denotes a full subcategory whose objects are such~$H$ that $p^*H\in\AA^{\perf}_k$.
\end{theorem}
\begin{proof}
According to Theorem~\ref{th_descentfor}, the category $\D^{\perf}(S)$ is equivalent to
$\D^{\perf}(X)_{\TT_p}$.  So we can use Theorem~\ref{th_sodt} about
semiorthogonal decompositions for descent categories.

Let us extend the semiorthogonal decomposition
$\D^{\perf}(X)=\langle\AA^{\perf}_1,\ldots,\AA^{\perf}_n\rangle$
to a semiorthogonal decomposition $\langle \AA_1,\ldots,\AA_n\rangle$ of
$\D(X)$, this is possible by Lemma~\ref{lemma_extension}.1.
Let us check that the functor $T_p=p^*p_*$ is upper triangular with respect to the decomposition $\langle
\AA_1,\ldots,\AA_n\rangle$.
One can show (see~\cite[prop. 4.2]{Ku}) that the category $\AA_i=(\AA^{\perf}_i)^{\oplus\infty}$ can be obtained from $\AA^{\perf}_i$
by adding arbitrary direct sums (once) and then consequently adding cones (many times). By hypothesis, for  $i<j$, $F_i\in\AA_i^{\perf}$,
$F_j\in\AA_j^{\perf}$ one has
$$\Hom(F_j,T_pF_i)=\Hom(F_j,p^*p_*F_i)=\Hom(F_j,p_{1*}p_2^*F_i)=\Hom(p_1^*F_j,p_2^*F_i)=0$$
(the second equality here is the flat base change formula).
For any families $F_i^{\a}\in\AA_i^{\perf}$,
$F_j^{\b}\in\AA_j^{\perf}$ one has
\begin{multline*}
\Hom\left(\bigoplus_{\b}F_j^{\b},T_p\left(\bigoplus_{\a}F_i^{\a}\right)\right)=
\prod_{\b}\Hom\left(F_j^{\b},\bigoplus_{\a}T_pF_i^{\a}\right)=\\
=\prod_{\b}\bigoplus_{\a}\Hom\left(F_j^{\b},T_pF_i^{\a}\right)=0,
\end{multline*}
because the objects $F_j$ are compact and the functor $T_p$ commutes with direct sums.
Taking cones does not spoil orthogonality, therefore the equality  $\Hom(F_j,T_pF_i)=0$ holds for all $F_i\in \AA_i, F_j\in \AA_j$.
It means that $T_p\AA_i\subset\langle\AA_1,\ldots,\AA_i\rangle$,
that is, $T_p$ is upper triangular.
To prove the theorem, we apply Theorem~\ref{th_sodt} to the categories $\CC=\D(X)$,
$\CC'=\D^{\perf}(X)$ and~$\AA_i$.
\end{proof}

\begin{theorem}
\label{th_sodschemescoh}
Let $X, S$ and $p$ be as in Theorem~\ref{th_sodschemesbig}.
Suppose that the bounded derived category of coherent sheaves $\D^b(X)$ has a semiorthogonal decomposition
$\langle\AA'_1,\ldots,\AA'_n\rangle$ into admissible subcategories.
Suppose that $\Hom(p_1^*F_j,p_2^*F_i)=0$ for all $1\le i<j\le n$ and objects $F_i\in \AA'_i,
    F_j\in\AA'_j$.
Then the category $\D^b(S)$ has  a semiorthogonal decomposition
$\langle\BB'_1,\ldots,\BB'_n\rangle$, where
$\BB'_k\subset\D^b(S)$ denotes a full subcategory, whose objects are such $H$ that $p^*H\in\AA'_k$.
\end{theorem}

\begin{proof}
The proof is analogous to the proof of the previous theorem,
Lemma~\ref{lemma_extension}.3 is used.
\end{proof}

\section{Semiorthogonal decompositions, invariant under the group action}
\label{section_invarianceofsod}

In this section we introduce invariant semiorthogonal decompositions of the category $\D^{\perf}(X)$ with respect to the action of an affine algebraic group~$G$ on a scheme $X$. We will show that the functor $p^*p_*$ (where $p$ is a canonical morphism $X\ra X\quot G$) is upper triangular with respect to a semiorthogonal decomposition $\D^{\perf}(X)=\langle\AA^{\perf}_1,\ldots,\AA^{\perf}_n\rangle$ if and only if $p_2^*\AA^{\perf}_i=a^*\AA^{\perf}_i$ for all
$i$ in a certain sense (as before, we use $a$ and $p_2\colon G\times X\to X$ to denote the action morphism and the projection onto the second factor).

\begin{lemma}
\label{lemma_bcdecomp} Let $X$ be a quasi-projective and $Y$ be an affine scheme over $\k$. Suppose the category of perfect complexes $\D^{\perf}(X)$ admits a semiorthogonal decomposition
$\langle\AA^{\perf}_1,\ldots,\AA^{\perf}_n \rangle$. Then the following semiorthogonal decompositions take place:
\begin{align*}
\D(X)&=\langle\AA_1,\ldots,\AA_n\rangle,\\
\D(Y\times X)&=\langle p_2^*\AA_1,\ldots,p_2^*\AA_n\rangle,\\
\D^{\perf}(Y\times X)&=\langle
p_2^*\AA^{\perf}_1,\ldots,p_2^*\AA^{\perf}_n\rangle,
\end{align*}
where
$\AA_i=(\AA^{\perf}_i)^{\oplus\infty}$ and the category  $p_2^*\AA_k$ is generated as a triangulated subcategory in
$\D(Y\times X)$ by objects of the form $p_2^*F_k$, $F_k\in\AA_k$.
\end{lemma}
\begin{proof}
Most statements of this lemma are proved in the paper by Kuznetsov~\cite{Ku} about the base change for semiorthogonal decompositions. We recall his constructions in order to prove the following fact, which is specific for the case of an affine scheme $Y$: the category
$p_2^*\AA_k$ is generated by objects of the form $p_2^*F_k$, $F_k\in\AA_k$.

Define
$p_2^*\AA^{\perf}_i$ to be the subcategory in $\D^{\perf}(Y\times X)$, generated by objects of the form $p_1^*H\otimes p_2^*F_i$, where $H\in
\D^{\perf}(Y)$ and $F_i\in\AA^{\perf}_i$, using shifts, cones and taking direct summands. By~\cite[5.1]{Ku}, we get a semiorthogonal decomposition
$$\D^{\perf}(Y\times X)=\langle p_2^*\AA_1^{\perf},\ldots,p_2^*\AA_n^{\perf}\rangle.$$

Define $\AA_i$ as $(\AA^{\perf}_i)^{\oplus\infty}$, define $p_2^*\AA_i$
as $(p_2^*\AA^{\perf}_i)^{\oplus\infty}$. According to~\cite[prop. 4.2]{Ku}, these categories form semiorthogonal decompositions
$$\D(X)=\langle \AA_1,\ldots,\AA_n\rangle,\qquad\D(Y\times X)=\langle p_2^*\AA_1,\ldots,p_2^*\AA_n\rangle.$$

Now let us check that the category $p_2^*\AA_i$ is generated by the objects of the form
$p_2^*F_i$, $F_i\in\AA_i$, as a triangulated category.
First, all such objects lie in $p_2^*\AA_i$.
Then,   consider the triangulated subcategory $\D'$ in  $\D(Y\times X)$,
generated by the objects $p_2^*F_i$, $F_i\in\AA_i$, for all $1\le i\le n$. This category is closed under taking arbitrary
direct sums because
all $\AA_i$ are closed and the functor $p_2^*$ commutes with direct sums. Therefore, $\D'$ is closed under taking direct summands: they can be
expressed via countable direct sums and cones.
The category~$\D'$ contains all objects of the form $p_2^*F_i$ for
$F_i\in\AA_i^{\perf}$ and hence, all objects of the form $p_2^*F=p_1^*\O_Y\otimes p_2^*F$ for $F\in
\D^{\perf}(X)=\langle\AA_1^{\perf},\ldots,\AA_n^{\perf}\rangle$.
Since the scheme $Y$ is affine, the right orthogonal to $\O_Y$ in   $\D(Y)$ is zero. By a result of Ravenel and Neeman (see~\cite[2.1.2]{BvdB}), the object
$\O_Y$ generates the category $\D^{\perf}(Y)$ by taking shifts, cones and direct summands. The category $\D'$ is closed under direct summands, and the functor
$p_1^*(-)\otimes p_2^*F$ preserves shifts, cones and direct summands. Therefore, all object of the form $p_1^*H\otimes p_2^*F$ with
$H\in\D^{\perf}(Y)$, $F\in\D^{\perf}(X)$, lie in $\D'$.
Hence, it follows from~\cite[lemma 5.2]{Ku} that $\D'$ contains $\D^{\perf}(Y\times X)$.
Finally, since the category $\D'$ is closed under direct sums, it coincides with $\D(Y\times X)$.
Now apply~\cite[lemma 3.2]{Ku} to see that the category $p_2^*\AA_i$ is generated by objects of the form $p_2^*F_i$, $F_i\in\AA_i$.
\end{proof}

Suppose that an affine algebraic group $G$ acts on a scheme~$X$, let $p$ be the canonical morphism $X\ra X\quot G$.
Note that two morphisms $p_2$ and $a\colon G\times X\ra X$ are isomorphic in the sense that there is a commutative diagram
$$\xymatrix{
 G\times X \ar[rr]^{(g,x)\mapsto (g^{-1},gx)}_{\sim} \ar[rd]_{a} &&
G\times X \ar[ld]^{p_2} \\ &X.&}$$
Therefore, the previous lemma can be applied also to either of the maps
$p_2$ and $a\colon G\times
X\ra X$. For the same reasons the lemma holds for the maps
${\mu}\colon G\times G\ra G$, $p_{23},\mu\times 1,1\times a\colon
G\times G\times X\ra G\times X$ and $p_3,ap_{23},a(1\times
a)=a(\mu\times 1)\colon G\times G\times X\ra X$.

\medskip
We introduce the following definition:
\begin{definition}
\label{def_invariantsubcat}
For an action of an algebraic group $G$ on a  scheme $X$
and a subcategory $\AA^{\perf}\subset \D^{\perf}(X)$,
define $p_2^*\AA^{\perf}$ to be the smallest full triangulated subcategory in $\Dp(G\times X)$, containing all objects of the form $p_1^*H\otimes p_2^*F$ where $H\in \Dp(G), F\in \AA^{\perf}$ and closed under direct summands.
Likewise, define $a^*\AA^{\perf}$ to be the smallest full triangulated subcategory in $\Dp(G\times X)$, containing all objects of the form $p_1^*H\otimes a^*F$ where $H\in \Dp(G), F\in \AA^{\perf}$ and closed under direct summands.

We say that $\AA^{\perf}$ is \emph{invariant} under the action if the subcategories
 $p_2^*\AA^{\perf}$ and $a^*\AA^{\perf}$ in $\D^{\perf}(G\times X)$
coincide. We will also say in this case that the action \emph{preserves} $\AA^{\perf}$.
\end{definition}

\begin{predl}
\label{prop_Gpreserves} Suppose that the category of perfect complexes on $X$ admits a semiorthogonal decomposition $\langle\AA^{\perf}_1,\ldots,\AA^{\perf}_n\rangle$, let $\D(X)=\langle\AA_1,\ldots,\AA_n\rangle$ be the corresponding decomposition
of the unbounded derived category.
Then the following conditions are equivalent:
\begin{enumerate}
\item for any $1\le i<j\le n$ and $F_i\in \AA^{\perf}_i, F_j\in\AA^{\perf}_j$, one has
$\Hom(p_2^*F_j,a^*F_i)=0$,
\item the functor $T_p=p^*p_*$ is upper triangular with respect to the decomposition $\D(X)=\langle\AA_1,\ldots,\AA_n\rangle$,
\item $p_2^*\AA_i=a^*\AA_i$ for all $i$,
\item $p_2^*\AA^{\perf}_i=a^*\AA^{\perf}_i$ for all $i$.
\end{enumerate}
\end{predl}
\noindent If these conditions are satisfied, we say that the action of the group \emph{preserves} the semiorthogonal decomposition.

\begin{proof}
1 $\Rightarrow$ 2 is, essentially, contained in the proof of Theorem~\ref{th_sodschemes} (with $p_2$ changed to $p_1$, and $a$ changed to $p_2$).

2 $\Rightarrow$ 3.
Take an object $F_k\in\AA_k$. For any  $i<k$ and
$F_i\in\AA_i$ we have $$\Hom(p_2^*F_k,a^*F_i)=\Hom(F_k,p_{2*}a^*F_i)=\Hom(F_k,T_pF_i)=0,$$
because
$T_pF_i\in \langle\AA_1,\ldots,\AA_i\rangle=\langle\AA_{i+1},\ldots,\AA_n\rangle^{\perp}$.
By Lemma~\ref{lemma_bcdecomp}, the objects $a^*F_i$ generate
$a^*\AA_i$ as a triangulated category, hence $p_2^*F_k\in\
^{\perp}a^*\AA_i$. Similarly, for $i>k$ and any $F_i\in\AA_i$ we have
$$\Hom(a^*F_i,p_2^*F_k)=\Hom(F_i,a_*p_2^*F_k)=\Hom(F_i,T_pF_k)=0.$$
As before, we obtain that $p_2^*F_k\in a^*\AA_i^{\perp}$. Finally,
$$p_2^*F_k\in ^{\perp}\langle a^*\AA_1,\ldots,a^*\AA_{k-1}\rangle \cap \langle a^*\AA_{k+1},\ldots,a^*\AA_{n}\rangle^{\perp}=a^*\AA_k.$$
By Lemma~\ref{lemma_bcdecomp}, the objects $p_2^*F_k$ generate the category $p_2^*\AA_k$,
hence $p_2^*\AA_k\subset a^*\AA_k$.
In the same way it is proved that $a^*\AA_k\subset p_2^*\AA_k$.

3 $\Rightarrow$ 4.
By Lemma~\ref{lemma_extension}.1,
$$p_2^*\AA^{\perf}_i=p_2^*\AA_i\cap\D^{\perf}(G\times X)=
a^*\AA_i\cap\D^{\perf}(G\times X)=a^*\AA^{\perf}_i.$$

4 $\Rightarrow$ 1. We have $p_2^*F_j\in p_2^*\AA^{\perf}_j=a^*\AA^{\perf}_j$ and $a^*F_i\in a^*\AA^{\perf}_i$.
Since $a^*\AA^{\perf}_i$ and  $a^*\AA^{\perf}_j$ are semiorthogonal, we get $\Hom(p_2^*F_j,a^*F_i)=0$.
\end{proof}

\begin{remark}
Note that conditions 1 and 2 refer to the whole semiorthogonal decomposition, while conditions 3 and  4 express invariance of a semiorthogonal decomposition through invariance of its components.
\end{remark}

\begin{remark}
There are analogs of Lemma~\ref{lemma_bcdecomp} and Proposition~\ref{prop_Gpreserves} for the bounded derived categories instead of the categories of perfect complexes, see~\cite[th. 5.6]{Ku} and Lemma~\ref{lemma_extension}.3.  The components of semiorthogonal decompositions in these analogs are supposed to be admissible.
\end{remark}

\begin{remark}
An action of a finite group $G$ preserves the decomposition $\D^{\perf}(X)=\langle\AA^p_1,\ldots,\AA^p_n\rangle$ if and only if
$g^*\AA^p_i=\AA^p_i$ for any $g\in G$ and all~$i$.
Indeed, in this case the product $G\times X$ is a disjoint union $\bigsqcup_{g\in G}X$ of several copies of
$X$, and the maps $p_2, a\colon G\times X\ra X$ have the form $(1_X)_{g\in
G}$ and $(g)_{g\in G}$ respectively. The subcategories $p_2^*\AA^p_i$ and
$a^*\AA^p_i$ in  $\D^{\perf}(G\times X)$ are the categories
$\bigoplus_{g\in G}\AA^p_i$ and $\bigoplus_{g\in G}g^*\AA^p_i$.  They coincide if and only if $g^*\AA^p_i=\AA^p_i$ for all $g$ and $i$.
\end{remark}

\section{Descent for semiorthogonal decompositions: equivariant categories}\label{section_descentforequiv}

In this section we prove main theorems that allow to construct a semiorthogonal decomposition of the derived category of equivariant sheaves on a scheme starting from a semiorthogonal decomposition of the derived category of sheaves that is preserved by the action.

For an object $F\in \D^G(X)$, denote by $\overline{F}\in \D(X)$
the object, obtained from~$F$ by forgetting of the equivariant structure.

\begin{theorem}
\label{th_sodforequivbig}
Suppose that an affine group scheme $G$ of finite type over a field   $\k$ acts on a quasi-projective scheme $X$ over $\k$. Suppose that $G$
is linearly reductive.
Suppose a semiorthogonal decomposition $$\D(X)=\langle\AA_1,\ldots,\AA_n\rangle$$
is given such that for any $1\le i<j\le n$ and $F_i\in \AA_i,\, F_j\in\AA_j$ one has
$\Hom(p_2^*F_j,a^*F_i)=0$. Denote by~$\BB_i$ the full subcategory in $\D^G(X)$, consisting of objects $F$ such that ${\overline F}\in \AA_i$:
$$\BB_i=\{F\in \D^G(X)\mid {\overline F}\in \AA_i\}.$$
Then there is a semiorthogonal decomposition $$\D^G(X)=\langle\BB_1,\ldots,\BB_n\rangle.$$
\end{theorem}
\begin{proof}
Let $p\colon X\ra X\quot G$ be the natural morphism of
stacks. By Theorem~\ref{th_descentfor2}, $\D^G(X)$ is equivalent to the descent category $\D(X)_{\TT_G}$ of comodules over the comonad $\TT_G=(p^*p_*,\e,\delta)$ on $\D(X)$. The proof of Theorem~\ref{th_sodschemes} implies that the functor $T_G=p^*p_*=p_{2*}a^*$ is upper triangular with respect to the initial semiorthogonal decomposition. Now Theorem~\ref{th_sodt} applied to the categories $\CC=\CC'=\D(X)$, the decomposition $\D(X)=\langle\AA_1,\ldots,\AA_n\rangle$ and the comonad $\TT_G$ proves our theorem.
\end{proof}

\begin{theorem}
\label{th_sodforequiv}
Let $X$ and  $G$ be as in Theorem~\ref{th_sodforequivbig}.
Suppose that in a semiorthogonal decomposition
$$\D^{\perf}(X)=\langle\AA^{\perf}_1,\ldots,\AA^{\perf}_n\rangle$$
for any $1\le i\le n$ one has $p_2^*\AA^{\perf}_i=a^*\AA^{\perf}_i$. Denote
$$\BB^{\perf}_i=\{F\in \D^{\perf,G}(X)\mid \overline F\in \AA^{\perf}_i\}.$$
Then there is a semiorthogonal decomposition $$\D^{\perf,G}(X)=\langle\BB^{\perf}_1,\ldots,\BB^{\perf}_n\rangle.$$
\end{theorem}
\begin{proof}
Following Lemma~\ref{lemma_extension}.1, extend the initial semiorthogonal decomposition to a semiorthogonal decomposition $\D(X)=\langle\AA_1,\ldots,\AA_n\rangle$ such that
$\AA^{\perf}_i=\AA_i\cap\D^{\perf}(X)$. By Proposition~\ref{prop_Gpreserves}, the functor $p^*p_*$ is upper triangular
with respect to  this extended decomposition. To prove out theorem, apply Theorem~\ref{th_sodt} to the categories $\CC=\D(X)$, $\CC'=\D^{\perf}(X)$, the decomposition $\D(X)=\langle\AA_1,\ldots,\AA_n\rangle$ and the comonad $\TT_G$.
\end{proof}

\begin{theorem}
\label{th_sodforequivcoh}
Let $X$ and $G$ be as in Theorem~\ref{th_sodforequivbig}.
Suppose that a semiorthogonal decomposition into admissible subcategories 
$$\D^b(X)=\langle\AA'_1,\ldots,\AA'_n\rangle$$
is given, and
for any $1\le i\le n$ one has $p_2^*\AA'_i=a^*\AA'_i$. Denote
$$\BB'_i=\{F\in \D^b(\coh^G(X))\mid \overline F\in \AA'_i\}.$$
Then there is a semiorthogonal decomposition
$$\D^b(\coh^G(X))=\langle\BB'_1,\ldots,\BB'_n\rangle.$$
\end{theorem}
\noindent
The proof is similar to the proof of Theorem~\ref{th_sodforequiv}, the analog of Proposition~\ref{prop_Gpreserves} for bounded derived categories is used.

\section{Derived descent theory for twisted equivariant sheaves}\label{section_twistedsheavesgeneralcase}

In this section we generalize results from~\cite{El_C4} about descent for equivariant derived categories to the case of twisted equivariant sheaves. We use the notions of cocycle on an algebraic group and of twisted equivariant sheaves, introduced in~\cite{El_C1}.
Let us recall necessary definitions and facts.

Let $G$ be a group scheme of finite type over a field $\k$ with the
structure morphism
$\mu\colon G\times G\ra G$.
\begin{definition}
\label{def_la}
A \emph{cocycle} on  $G$ is a pair $(\LL,\a)$ consisting of a line bundle $\LL$ on $G$ and of an isomorphism of bundles $\a\colon p_1^*\LL\otimes p_2^*\LL\ra\mu^*\LL$
on $G\times G$, satisfying the associativity condition:
the diagram of isomorphisms of sheaves
on the product $G\times G\times G$
$$\xymatrix{
p_1^*\LL\otimes p_2^*\LL\otimes p_3^*\LL \ar[rrr]^{ 1\otimes p_{23}^*\a} \ar[d]_{p_{12}^*\a\otimes  1}&&&
p_1^*\LL\otimes (\mu p_{23})^*\LL \ar[d]^{( 1\times \mu)^*\a} \\
(\mu p_{12})^*\LL\otimes p_3^*\LL \ar[rr]^{(\mu\times  1)^*\a} &&
(\mu(\mu\times  1))^*\LL & (\mu( 1\times \mu))^*\LL \ar@{-}[l]_{\sim}
}
$$
is commutative.
\end{definition}

Let $(\LL,\a)$ be a cocycle on $G$.
Suppose the group $G$ acts on a scheme~$X$ and $a\colon G\times X\ra X$ denotes the action map.
\begin{definition}
\label{def_Lasheaf}
We say that an \emph{$(\LL,\a)$-$G$-equivariant sheaf}
on $X$ is a pair $\FF=(F,\theta)$ of a sheaf $F$ on $X$ and an isomorphism
$$\theta\colon p_1^*\LL\otimes p_2^*F\ra a^*F$$
of sheaves on $G\times X$,
satisfying the compatibility condition:
the diagram of sheaves on $G\times G \times X$
$$\xymatrix{
p_1^*\LL\otimes p_2^*\LL\otimes p_3^*F \ar[rrr]^{ 1\otimes p_{23}^*\theta} \ar[d]_{p_{12}^*\a\otimes  1}&&&
p_1^*\LL\otimes (ap_{23})^*F \ar[d]^{( 1\times a)^*\theta} \\
(\mu p_{12})^*\LL\otimes p_3^*F \ar[rr]^{(\mu\times  1)^*\theta} &&
(a(\mu\times  1))^*F & (a( 1\times a))^*F \ar@{-}[l]_{\sim}
}
$$
is commutative.
A \emph{morphism} of $(\LL,\a)$-equivariant sheaves from $(F_1,\theta_1)$ to$(F_2,\theta_2)$ is a homomorphism of sheaves $F_1\ra F_2$ compatible  with the structure isomorphisms $\theta_1$ and~$\theta_2$.
\end{definition}

As a special case when $X$ is a point, we get a definition of an $(\LL,\a)$-representation of the group $G$.

We denote the abelian category of quasi-coherent $(\LL,\a)$-$G$-equivariant sheaves on $X$ by
$\qcoh^{G,\LL,\a}(X)$, and the category of coherent $(\LL,\a)$-$G$-equivariant sheaves on $X$ by  $\coh^{G,\LL,\a}(X)$.

There is a well-defined tensor multiplication on cocycles on the group~$G$, it makes the set of isomorphism classes of cocycles on $G$ an abelian group.
We denote tensor powers of a cocycle $(\LL,\a)$ by $(\LL^{r},\a^{r})$.
As in the case of finite groups, cocycles classify central extensions of a given group by $\Gm$, the multiplicative group of the field.
That is, there is a bijection between isomorphism classes of cocycles on an algebraic group~$G$ over $\k$ and isomorphism classes of central extensions of $G$ by $\Gm$:
\begin{equation*}
1\ra \Gm\ra \~G\xra{\pi} G\ra 1.
\end{equation*}

Twisted representations and twisted equivariant sheaves can be described with the help  of the extension of the group, corresponding to the given cocycle.

\begin{predl}[{\cite[prop. 1.9]{El_C1}}]
\label{prop_aboutHH}
In our assumptions, for any integer~$r$
there is an equivalence of categories 
$$\qcoh^{G,\LL^r,\a^r}(X) \cong \qcoh_{(r)}^{\~G}(X).$$
Here
$\qcoh_{(r)}^{\~G}(X)$ denotes a full subcategory in $\qcoh^{\~G}(X)$ formed by $\~G$-equivariant sheaves $F$ such that the subgroup $\Gm\subset \~G$ acts on $F$ with the weight $r$.
The similar statement holds for coherent sheaves.
\end{predl}

The subgroup $\Gm\subset \~G$
is central and acts trivially on $X$, therefore the following decompositions over characters of $\Gm$ take place:
$$\coh^{\~G}(X)=\bigoplus\nolimits_{r\in\Z}\coh^{\~G}_{(r)}(X),\qquad
\qcoh^{\~G}(X)=\prod\nolimits_{r\in\Z}\qcoh^{\~G}_{(r)}(X).$$
That is, the categories of twisted $G$-equivariant sheaves are equivalent to full subcategories (and, moreover, to direct factors)
of the category of equivariant sheaves with respect to a certain extension of the group  $G$ by $\Gm$.

\medskip
As well as usual equivariant sheaves, twisted equivariant sheaves can be viewed as objects of a descent category, related to a certain cosimplicial category. \begin{definition}
\label{def_twistedp}
Consider simplicial scheme~(\ref{equation_X/G})
$$[X,G\times X,G\times G\times X,\ldots,p_{\bul}].$$
Consider cosimplicial category 
$$[\qcoh(X),\qcoh(G\times X),\qcoh(G\times  G\times X),\ldots,p_{\bul}^*],$$
that consists of categories of quasi-coherent sheaves and pullback functors between them. Twist the functors $p_f^*$ as follows: for  a sheaf $F$ on $\underbrace{G\times \ldots\times G}_{m-1\ \text{times}}\times X$
and  a map $f\colon [1,\ldots,m]\to [1,\ldots n]$ we set 
\begin{equation}
P_f'^*F=p_1^*\LL\otimes\ldots\otimes p_{n-f(m)}^*\LL\otimes p_f^*F.
\end{equation}
Here $p_i, i=1,\ldots,n-f(m)$ denotes the projection of $G\times \ldots\times G\times X$
onto the $i$-th factor~$G$.

Now we need to define isomorphisms $\epsilon_{f,g}'\colon P_f'^*P_g'^*\xra{\sim} P_{fg}'^*$ for any pair of composable maps $f, g$.  Let $f\colon [1,\ldots,n]\ra [1,\ldots,k]$
and $g\colon [1,\ldots,m]\ra [1,\ldots,n]$ be morphisms in $\Delta$.
We have 
\begin{align*}
P'^*_fP'^*_g(-)&=P'^*_f\left(p_1^*\LL\otimes \ldots \otimes p_{n-g(m)}^*\LL\otimes p_g^*(-)\right)=\\
&=p_1^*\LL\otimes \ldots \otimes p_{k-f(n)}^*\LL\otimes
p_f^*p_1^*\LL\otimes\ldots\otimes p_f^*p_{n-g(m)}^*\LL\otimes
p_f^*p_g^*(-)=\\
&=p_1^*\LL\otimes \ldots \otimes p_{k-f(n)}^*\LL\otimes
(\mu p_{k-f(n)+1,\ldots,k-f(n-1)})^*\LL\otimes\ldots\\
&\quad\ldots\otimes
(\mu p_{k-f(g(m)+1)+1,\ldots,k-f(g(m))})^*\LL\otimes
p_f^*p_g^*(-),\\
P'^*_{fg}(-)&=p_1^*\LL\otimes\ldots\otimes p_{k-f(g(m))}^*\LL\otimes p_{fg}^*(-),
\end{align*}
where $p_{i,i+1,\ldots,j-1,j}$ denotes the projection of $G\times \ldots\times G\times X$ onto the product of $i,\ldots,j$-th factors, and $\mu$ denotes the multiplication $G\times\ldots\times G\ra G$.
Define the isomorphism $\epsilon_{f,g}'$ via the isomorphism $\epsilon_{f,g}\colon p_f^*p_g^*\ra p_{fg}^*$ and the isomorphisms
$$(\mu p_{i,i+1,\ldots,j-1,j})^*\LL\ra p_i^*\LL\otimes\ldots\otimes p_j^*\LL$$
obtained by iterating the structure isomorphism 
$\a\colon p_1^*\LL\otimes p_2^*\LL\ra \mu^*\LL$ on $G\times G$.

The resulting cosimplicial categories 
\begin{gather}
\label{equation_xgla}
[\qcoh(X),\qcoh(G\times X),\qcoh(G\times  G\times X),\ldots,P_{\bul}'^*],\\
\label{equation_xglabis}
[\D(X),\D(G\times X),\D(G\times  G\times X),\ldots,P_{\bul}'^*],
\end{gather}
and similar to them will be called  
cosimplicial categories, \emph{associated with the action of $G$ on $X$, twisted in the cocycle $(\LL,\a)$.}
\end{definition}

It can be directly seen from Definition~\ref{def_dxbar} of $\Kern$ 
and from the above construction that the descent category $\Kern$, associated with~\eqref{equation_xgla}, is equivalent to the category of $(\LL,\a)$-$G$-equivariant quasi-coherent sheaves on $X$.

\begin{predl}
\label{prop_twistedbc}
Categories~\eqref{equation_xgla}~and~\eqref{equation_xglabis} are cosimplicial categories with base change in the sense of Definition~\ref{def_basechangecat}.
\end{predl}
\begin{proof}
One has to verify two conditions: the functors $P'^*_{\bul}$ have right adjoint functors and the canonical base change morphisms are isomorphisms for exact Cartesian squares.

One can check that the functor $$P'^*_{f}=p_1^*\LL\otimes\ldots\otimes p_{n-f(m)}^*\LL\otimes p_f^*(-)$$ has a right adjoint functor $$P'_{f*}=p_{f*}(p_1^*\LL^{-1}\otimes\ldots\otimes p_{n-f(m)}^*\LL^{-1}\otimes -).$$
Indeed,
\begin{multline*}
\Hom(p_1^*\LL\otimes\ldots\otimes p_{n-f(m)}^*\LL\otimes p_f^*F_1,F_2)\cong \\
\Hom(p_f^*F_1,p_1^*\LL^{-1}\otimes\ldots\otimes p_{n-f(m)}^*\LL^{-1}\otimes F_2)\cong\\
\Hom(F_1,p_{f*}(p_1^*\LL^{-1}\otimes\ldots\otimes p_{n-f(m)}^*\LL^{-1}\otimes F_2)).
\end{multline*}

The second condition says that for a commutative square of non-decreasing maps
of finite sets 
$$\xymatrix{[1,\ldots,m+n-r]&[1,\ldots,n] \ar[l]_-{f'}\\
[1,\ldots,m] \ar[u]_{g'} & [1,\ldots,r], \ar[l]_{f}
\ar[u]_g}$$
with injective $f$ and $f'$ and $[1,\ldots,m+n-r]=\mathrm{Im}\ f'\cup \mathrm{Im}\ g'$, the base change morphism 
$P'^*_fP'_{g*}\ra P'_{g'*}P'^*_{f'}$ is an isomorphism of functors. 
Decomposing $f$ and  $g$ into the composition of basic maps that correspond to 
faces and degenerations, we come down to several simple cases. 
That is, one can assume that $m=r+1$, $f=\delta_i$ (where $\delta_i(x)=x$ for $x<i$ and $\delta_i(x)=x+1$ for $x\ge i$) and that $n=r+1$, $g=\delta_j$ or $n=r-1$, $g=s_j$ (where $s_j(x)=x$ for $x\le j$ and $s_j(x)=j-1$ for $x>j$). Essentially, 
the cases differ depending on whether $i$ and  $j$ equals $1$. In every case the 
arguments are straightforward and based on using adjunction formula, the isomorphism~$\a$ and standard base change formula.
Suppose, for example, that $f=g=\delta_1$. Then either $f'=\delta_1$,
$g'=\delta_2$ or  $f'=\delta_2$,
$g'=\delta_1$. Consider the first case.

By Definition~\ref{def_twistedp} and the first part of the proof one has 
\begin{align*}
P'_{g'*}P'^*_{f'}&=p_{g'*}(p_1^*\LL\otimes p_{f'}^*(-)) &&\text{by the definition}\\
& \cong
p_{g'*}(p_1^*\LL\otimes p_2^*\LL\otimes p_2^*\LL^{-1}\otimes p_{f'}^*(-)) &&\\
& \cong
p_{g'*}((\mu p_{12})^*\LL\otimes p_{f'}^*p_1^*\LL^{-1}\otimes p_{f'}^*(-))
&&\\
\intertext{we used the isomorphism $\a\colon p_1^*\LL\otimes p_2^*\LL\to
\mu^*\LL$ of bundles on $G\times G$ and the equality $p_1p_{f'}=p_2$}
& \cong
p_{g'*}(p_{g'}^*p_1^*\LL\otimes p_{f'}^*(p_1^*\LL^{-1}\otimes -))&&\\
\intertext{because $p_{g'}(g_{r+2},g_{r+1}, \ldots, g_2,x_1)=(g_{r+2}g_{r+1},
\ldots, g_2,x_1)$ and $p_1p_{g'}=\mu p_{12}$}
& \cong
p_1^*\LL\otimes p_{g'*}p_{f'}^*(p_1^*\LL^{-1}\otimes -))&&
\text{by the projection formula for\ } p_{g'}\\
& \cong
p_1^*\LL\otimes p_f^*p_{g*}(p_1^*\LL^{-1}\otimes -))&&\\
\intertext{by the flat base change formula for the morphisms $p_f$ and $p_g$}
& = P'^*_fP'_{g*} &&\text {by the definition.}
\end{align*}
The other cases are treated in a similar way.
\end{proof}

Denote by $T_{G,\LL,\a}$ the comonads on $\qcoh(X)$ and $\D(X)$ associated with 
the cosimplicial categories \eqref{equation_xgla} and \eqref{equation_xglabis},
see Definition~\ref{def_simpl>comonad}.
By Propositions~\ref{prop_simpl>comonad} and~\ref{prop_twistedbc},
the descent categories 
$$\qcoh(X)^{G,\LL,\a}=\Kern([\qcoh(X),\qcoh(G\times X),\qcoh(G\times  G\times X),\ldots,P_{\bul}'^*])$$
and
$$\D(X)^{G,\LL,\a}=\Kern([\D(X),\D(G\times X),\D(G\times  G\times X),\ldots,P_{\bul}'^*])$$
are equivalent to the categories of comodules 
$\qcoh(X)_{T_{G,\LL,\a}}$ and $\D(X)_{T_{G,\LL,\a}}$ respectively.

\begin{example}
For small values of $n,m$ and for $f$, corresponding to faces and degenerations, the functors $P_f'^*$ from Definition~\ref{def_twistedp} have the form
$$\xymatrix{
{\qcoh(G\times G\times X)} \ar@<-5mm>[rr]^-{(e\times 1)^*}
\ar@<5mm>[rr]^-{(p_1\times e\times p_2)^*} &&
{\qcoh(G\times X)} \ar@<10mm>[ll]_-{p_1^*\LL\otimes p_{23}^*} \ar[ll]_-{(\mu\times 1)^*} \ar@<-10mm>[ll]_-{(1\times a)^*} \ar[rr]^-{(e\times 1)^*}
&& {\qcoh(X)} \ar@<5mm>[ll]_-{p_1^*\LL\otimes p_2^*} \ar@<-5mm>[ll]_-{a^*}} .$$
For $g\colon [1]\ra [1,2], g(1)=1$ and $f\colon [1,2]\ra [1,2,3], f(1)=1, f(2)=3$ the isomorphism between the functors $P_f'^*P_g'^*$ and $P_{fg}'^*$ is the following one:
\begin{multline*}P_f'^*P_g'^*({-})=(\mu\times 1)^*(p_1^*\LL\otimes p_2^*({-}))\xra{\sim}
p_{12}^*\mu^*\LL\otimes p_3^*({-})\xra{p_{12}^*(\a)^{-1}}\\
\ra p_{12}^*(p_1^*\LL\otimes p_2^*\LL)\otimes p_3^*({-})\xra{\sim}
p_1^*\LL\otimes p_2^*\LL\otimes p_3^*({-})
=P_{fg}'^*({-}).
\end{multline*}
\end{example}

It follows from the definition that the descent category $\Kern$, 
associated with~(\ref{equation_xgla}), is equivalent to the category of twisted 
$(\LL,\a)$-$G$-equivariant quasi-coherent sheaves on~$X$.

Remark that the cosimplicial category~(\ref{equation_xgla})
cannot be obtained from a morphism of stacks via the standard construction 
from Example~\ref{example_standard}.
 
Now we come to the derived descent for twisted sheaves.
Let $\D^{G,\LL,\a}(X)=\D(\qcoh^{G,\LL,\a}(X))$ be the unbounded derived category 
of $(\LL,\a)$-$G$-equivariant quasi-coherent sheaves on $X$.
Denote by $\D^b(\coh^{G,\LL,\a}(X))$ the bounded derived category of
$\coh^{G,\LL,\a}(X)$. Denote by $\D^{\perf,G,\LL,\a}(X)$ the category of $(\LL,\a)$-$G$-equivariant perfect complexes. This is the full subcategory in 
$\D^{G,\LL,\a}(X)$ formed by objects that are perfect complexes on $X$ after forgetting of the group action.

As in the non-twisted context, we aim to show that for a linearly reductive group 
$G$ the derived category of twisted equivariant sheaves $\D^{G,\LL,\a}(X)$ 
is equivalent to the descent category 
$$\D(X)^{G,\LL,\a}=\Kern([\D(X),\D(G\times X),\D(G\times G\times X),\ldots,P_{\bul}'^*]),$$
related to the cosimplicial category from Definition~\ref{def_twistedp}.
We deduce this from the fact that the comparison functor is an equivalence for the derived category of equivariant with respect to a certain extension of the group $G$ sheaves on $X$.
Let 
\begin{equation}
\label{equation_groupextension2}
1\ra\Gm\ra\~G\xra{\pi} G\ra 1
\end{equation}
be the central extension of the group $G$ by $\Gm$ corresponding to the cocycle
$(\LL,\a)$.
The category of quasi-coherent $\~G$-equivariant sheaves on $X$ is decomposed into the direct product of categories 
\begin{equation}
\label{equation_decompovercocycles}
\qcoh^{\~G}(X)\cong\prod_{i\in \Z}\qcoh^{\~G}_{(i)}(X)\cong\prod_{i\in \Z}\qcoh^{\~G,\LL^i,\a^i}(X),
\end{equation}
where $\qcoh^{\~G}_{(i)}$  denotes the subcategory consisting of those
$\~G$-equivariant sheaves, on which the subgroup $\Gm\subset \~G$ 
acts with the weight $i$, see Proposition~\ref{prop_aboutHH}.

\begin{theorem}
\label{th_descentforequivtwist}
Let $X$ be a quasi-projective scheme over a field $\k$,
acted by an affine group scheme~$G$ of finite type over $\k$. 
Suppose that the category of representations of~$G$ over $\k$ is semisimple. 
Then the comparison functors 
\begin{align*}
\D^{G,\LL,\a}(X)&\ra \D(X)^{G,\LL,\a},\\
\D^{\perf,G,\LL,\a}(X)&\to \D^{\perf}(X)^{G,\LL,\a},\\
\D^b(\coh^{G,\LL,\a}(X))&\to \D^b(X)^{G,\LL,\a}
\end{align*}
are equivalences. Here 
$\D^{\perf}(X)^{G,\LL,\a}$ and $\D^b(X)^{G,\LL,\a}$ denote the subcategories in the descent category
$$\D(X)^{G,\LL,\a}=\Kern([\D(X),\D(G\times X),\D(G\times G\times X),\ldots,P_{\bul}'^*]),$$ associated with the subcategories $\D^{\perf}(X), \D^b(X) \subset \D(X)$.
\end{theorem}

\begin{proof}
Denote by 
$$\D_{(i)}\subset \D(X)^{\~G}=\Kern([\D(X),\D(\~G\times X),\D(\~G\times \~G\times X),\ldots,p_{\bul}^*])$$
the full subcategory, formed by those objects, on which the subgroup 
$\Gm\subset \~G$ acts with the weight $i$.
These subcategories are pairwise orthogonal, i.e.,
all morphisms between objects of different subcategories are zero.
Indeed, suppose $(F_1,\theta_1),(F_2,\theta_2)$ are the objects in $\D_{(i)},\D_{(j)}$ respectively, $i\ne j$,
and $f\colon (F_1,\theta_1)\to (F_2,\theta_2)$ is a morphism. It means that we have a commutative diagram of morphisms on $\~G\times X$:
$$\xymatrix{p_2^*F_1 \ar[rr]^{\theta_1}\ar[d]^{p_2^*f}&& a^*F_1\ar[d]^{a^*f}\\
p_2^*F_2 \ar[rr]^{\theta_2}&& a^*F_2.}$$
Restricting to $\Gm\times X\subset \~G\times X$, we get 
$$\xymatrix{p_2^*F_1 \ar[rr]^{\theta_1=t^i}\ar[d]^{p_2^*f}&& p_2^*F_1\ar[d]^{p_2^*f}\\
p_2^*F_2 \ar[rr]^{\theta_2=t^j}&& p_2^*F_2.}$$
We used that the maps $p_2$ and $a\colon \Gm\times X\to X$ coincide. Also, since $F_1\in \D_{(i)}$, the map $\theta_1$ in the above diagram is the multiplication in the function 
$t^i\in \Gamma(\Gm\times X,\O)$, and $\theta_2$ is the multiplication in $t^j$.
Since $p_2^*f$ is  a morphism in $\D(\Gm\times X)$, it commutes with multiplication in functions on $\Gm\times X$. Therefore we get 
$$p_2^*f\cdot (t^i-t^j)=0,$$
it implies that $f=0$.

The group extension~(\ref{equation_groupextension2}) induces 
a fully faithful functor on the descent categories 
$\D(X)^{G,\LL^i,\a^i}\ra\D(X)^{\~G}$ whose image lies in the subcategory~$\D_{(i)}$, see~\cite[proof of Prop. 1.9]{El_C1}. Denote the resulting functor 
$\D(X)^{G,\LL^i,\a^i}\ra\D_{(i)}$ by $\Psi_i$.
Decomposition~(\ref{equation_decompovercocycles}) yields the decomposition
\begin{equation*}
\D^{\~G}(X)\cong\prod_{i\in \Z}\D^{\~G,\LL^i,\a^i}(X).
\end{equation*}
for derived categories. 
It is included into the commutative diagram of categories and functors
$$\xymatrix{
{\prod_i\D^{G,\LL^i,\a^i}(X)} \ar[rrrr]^{\sim} \ar[d]^{\prod\Phi_i} &&&&{\D^{\~G}(X)} \ar[d]^{\Phi_{\~G}}_{\sim} \\
{\prod_i \D(X)^{G,\LL^i,\a^i}} \ar[rr]^{\prod\Psi_i} && {\prod_i\D_{(i)}} \ar[rr] && {\D(X)^{\~G}.}
}$$
The group $\~G$ is linearly reductive as an extension of $G$ by a torus.
Hence, the comparison functor $\Phi_{\~G}$ is an equivalence (see Theorem~\ref{th_descentfor2}). 
Since the subcategories $\D_{(i)}\subset \D(X)^{\~G}$ are orthogonal, the natural functor $\prod_i\D_{(i)}\ra \D(X)^{\~G}$ is fully faithful. 
Therefore, every composition of functors 
$$\D^{G,\LL^i,\a^i}(X)\xra{\Phi_i} \D(X)^{G,\LL^i,\a^i}\xra{\Psi_i} \D_{(i)}$$ 
is an equivalence as well. 
Now use that $\Psi_i$ is fully faithful to conclude that $\Phi_i$ is an  equivalence.

The statements about $\D^{\perf,G,\LL,\a}(X)$ and $\D^b(\coh^{G,\LL,\a}(X))$
are straightforward corollaries.
\end{proof}

\section{Functors between descent categories}

Results of Section~\ref{section_descentforequiv} allow to construct 
semiorthogonal decompositions of equivariant derived categories. 
Later we will describe components of these decompositions in three interesting examples.
In this section necessary technique is developed.

All schemes are supposed to be quasi-projective over a field $\k$.
By a group $G$ we understand an affine group scheme of finite type over~$\k$.

Let $X$ and $Y$ be schemes, acted by a group $G$, let $(\LL_X,\a_X)$ and
$(\LL_Y,\a_Y)$ be cocycles on the group $G$, let $\Psi\colon \D(X)\ra\D(Y)$
be a functor.
\begin{definition}
Suppose that there exists a functor $\Psi_{\bul}$ between 
cosimplicial categories $[\D(X),\D(G\times
X),\ldots,P_{X\bul}'^*]$ and $[\D(Y),\D(G\times
Y),\ldots,P_{Y\bul}'^*]$, which are associated with the action of $G$ on $X$,
twisted into cocycle $(\LL_X,\a_X)$, and the action of $G$ on $Y$,
twisted into cocycle $(\LL_Y,\a_Y)$ (see Definition~\ref{def_twistedp}), 
such that $\Psi_0=\Psi$. Then we say that the functor $\Psi$ \emph{is compatible with 
the (twisted) actions} of $G$ on $X$ and $Y$. 

Compatible functors between other versions of derived category or between abelian categories 
are defined analogously.
\end{definition}
By Lemma~\ref{lemma_psipsi}, if the functor 
$\Psi\colon\D(X)\ra\D(Y)$ is compatible with the actions of $G$ on $X$ and $Y$, twisted into cocycles $(\LL_X,\a_X)$ and $(\LL_Y,\a_Y)$, then it is compatible with the comonads $\TT_{G,\LL_X,\a_X}$ and
$\TT_{G,\LL_Y,\a_Y}$. Hence, by Lemma~\ref{lemma_fffunctor}, 
the functor $\Psi$ induces a functor on descent categories: 
 $\D(X)^{G,\LL_X,\a_X}\ra\D(Y)^{G,\LL_Y,\a_Y}$.

\begin{lemma}
\label{lemma_equivmorph}
Suppose $f\colon X\ra Y$ is an equivariant morphism of schemes over $\k$, 
acted by a group $G$, and $(\LL,\a)$ is a cocycle on $G$. 
Then the pullback functors $\qcoh(Y)\ra\qcoh(X)$,
$\D(Y)\ra\D(X)$, $\D^{\perf}(Y)\ra\D^{\perf}(X)$ are compatible with the action of $G$, twisted into $(\LL,\a)$.
Consequently, there are well-defined pullback functors between descent categories:
\begin{gather*}
\qcoh^{G,\LL,\a}(Y)\xra{f^*} \qcoh^{G,\LL,\a}(X),\\
\D(Y)^{G,\LL,\a}\xra{f^*} \D(X)^{G,\LL,\a},\\
\D^{\perf}(Y)^{G,\LL,\a}\xra{f^*} \D^{\perf}(X)^{G,\LL,\a},\\
\intertext{and for a morphism $f$ of finite $\Tor$-dimension,}
\D^b(Y)^{G,\LL,\a}\xra{f^*} \D^b(X)^{G,\LL,\a}.
\end{gather*}
The analogous statements hold for the pushforward functors 
$\qcoh(X)\ra\qcoh(Y)$ and $\D(X)\ra\D(Y)$.
\end{lemma}
\begin{proof}
Evidently, the functors 
$$\Psi_k=(\underbrace{1\times 1\times\ldots\times 1}_{k\ \text{times}}\times f)^*$$
and the canonical isomorphisms of the form $s^*t^*\cong (ts)^*$
define correctly a pullback functor between cosimplicial categories 
\begin{gather}
\label{cats_DYGYDXGX}
[\D(Y),\D(G\times Y),\ldots,P_{Y\bul}'^*]\ra
[\D(X),\D(G\times X),\ldots,P_{X\bul}'^*],
\end{gather}
associated with the twisted action of $G$ on $X$ and $Y$.

For the categories $\qcoh, \D^{\perf}, \D^b$ and for pushforward functors the 
proof is analogous.
\end{proof}

\begin{lemma}
\label{lemma_tensorE}
Let $\EE $ be an object of the category $\D^{G,\LL_0,\a_0}(X)$. Then  tensor multiplication by $\EE $ is compatible with the action of $G$ on $X$, twisted into the cocycle $(\LL,\a)$, and the action of $G$ on $X$, twisted into the cocycle $(\LL\otimes \LL_0,\a\otimes \a_0)$. It induces 
the functor $$\D(X)^{G,\LL,\a}\ra\D(X)^{G,\LL\otimes \LL_0,\a\otimes \a_0}.$$

If, in addition, $\EE\in \D^{\perf,G,\LL_0,\a_0}(X)$, then also the functors $$\D^{\perf}(X)^{G,\LL,\a}\ra\D^{\perf}(X)^{G,\LL\otimes \LL_0,\a\otimes \a_0}
\qquad\text{and}\qquad \D^b(X)^{G,\LL,\a}\ra\D^b(X)^{G,\LL\otimes \LL_0,\a\otimes \a_0}$$
are induced.
\end{lemma}
\begin{proof}
Consider the cosimplicial categories 
$$[\D(X),\D(G\times X),\D(G\times G\times X),\ldots,P_{\bul}'^*]$$
and
$$[\D(X),\D(G\times X),\D(G\times G\times X),\ldots,P_{\bul}''^*],$$
associated with the actions of $G$ on $X$, twisted into the cocycles 
$(\LL,\a)$ and $(\LL\otimes \LL_0,\a\otimes\a_0)$ respectively, see Definition~\ref{def_twistedp}. Let $E$ be an object in $\D(X)$, obtained from $\EE$ by forgetting of the equivariant structure. 
Define the functors 
$$\Psi_k\colon \D(\underbrace{G\times\ldots\times G}_{k \text{ times}}\times X)\ra \D(\underbrace{G\times\ldots\times G}_{k \text{ times}}\times X),$$
by the formula 
$$\Psi_k(-)=p_1^*\LL_0\otimes\ldots\otimes p_k^*\LL_0\otimes (-) \otimes p_{k+1}^*E .$$
Define the isomorphism of functors 
$$\b_f\colon \Psi_n\circ P_f'^*\ra P_f''^*\circ \Psi_m$$
for any map $f\colon [1,\ldots,m+1]\ra [1,\ldots,n+1]$ in $\Delta$.
We have 
\begin{multline*}
\Psi_n\circ P_f'^*(-)=p_1^*\LL_0\otimes\ldots\otimes p_n^*\LL_0\otimes P_f'^*(-)\otimes p_{n+1}^*E =\\
=p_1^*\LL_0\otimes\ldots\otimes p_n^*\LL_0\otimes
p_1^*\LL\otimes\ldots\otimes p_{n+1-f(m+1)}^*\LL\otimes p_f^*(-)\otimes
p_{n+1}^*E ,
\end{multline*}
\begin{multline*}
P_f''^*\circ\Psi_m(-)=\\
=p_1^*(\LL\otimes\LL_0)\otimes\ldots\otimes p_{n+1-f(m+1)}^*(\LL\otimes\LL_0)\otimes p_f^*(p_1^*\LL_0\otimes\ldots\otimes p_m^*\LL_0\otimes - \otimes p_{m+1}^*E )\cong\\
\cong p_1^*\LL_0\otimes\ldots\otimes p_{n+1-f(m+1)}^*\LL_0\otimes
(\mu p_{n+1-f(m+1)+1,\ldots,n+1-f(m)})^*\LL_0\otimes\ldots\\
\ldots\otimes
(\mu p_{n+1-f(2)+1,\ldots,n+1-f(1)})^*\LL_0
\otimes p_1^*\LL\otimes\ldots\otimes p_{n+1-f(m+1)}^*\LL\otimes
p_f^*(-)\otimes\\
\otimes (ap_{n+1-f(1)+1,\ldots,n+1})^*E ,
\end{multline*}
where $p_{i,i+1,\ldots,j-1,j}$ denotes the projection of $G\times \ldots\times G\times X$ onto the product of $i,\ldots,j$-th factors,  $\mu$ denotes the multiplication $G\times\ldots\times G\ra G$, and
$a$ denotes the iterated action morphism
$G\times\ldots\times G\times X\ra X$.
To define $\b_f$, we use the isomorphisms 
$$(\mu p_{i,i+1,\ldots,j-1,j})^*\LL_0\ra p_i^*\LL_0\otimes\ldots\otimes p_j^*\LL_0,$$
obtained by iterating of the structure isomorphism 
$\a_0\colon p_1^*\LL_0\otimes p_2^*\LL_0\ra \mu^*\LL_0$ on $G\times G$,
and the isomorphism 
$$(ap_{n+1-f(1)+1,\ldots,n+1})^*E \ra p_{n+1-f(1)+1}^*\LL_0\otimes\ldots\otimes p_n^*\LL_0\otimes p_{n+1}^*E ,$$
obtained by consequent applying of the structure isomorphism 
$\theta\colon p_1^*\LL_0\otimes p_2^*E \ra a^*E $ on $G\times X$. 
The cocycle conditions for $\a_0$ and the compatibility of $\a_0$ with $\theta$
imply that the functors $\Psi_k$ and the isomorphisms $\b_f$ do define a functor 
between the cosimplicial categories. Therefore, the functor $\Psi_0=-\otimes E
\colon \D(X)\ra \D(X)$ is compatible with the twisted actions of $G$ on
$X$.

By the arguments from the beginning of this section, we get 
a functor $\D(X)^{G,\LL,\a}\ra \D(X)^{G,\LL\otimes \LL_0,\a\otimes \a_0}$ on the descent
categories. For the categories $\D^{\perf}(X)$ and $\D^b(X)$ the proof is analogous.
\end{proof}

\begin{predl}
\label{prop_kernel}
Let $X$ and $Y$ be quasi-projective schemes over $\k$, acted by a group scheme $G$
of finite type over $\k$. Let 
$\Psi_{E}\colon\D(X)\ra \D(Y)$ be the functor, defined by the kernel 
$E\in\D(X\times Y)$:
$$\Psi_{E}(-)=p_{2*}(p_1^*(-)\otimes E).$$
Suppose that there is an object  $\EE\in\D^{G,\LL_0,\a_0}(X\times Y)$, which is 
$E$ with the action of $G$ forgotten, suppose  $(\LL,\a)$ is a cocycle.

1. 
Then the functor $\Psi_{E}$ is compatible with the action of $G$ on $X$, twisted into $(\LL,\a)$, and the action of $G$ on $Y$, twisted into $(\LL\otimes
\LL_0,\a\otimes\a_0)$.

2. 
Let $\TT_{G,\LL,\a}$ and  $\TT_{G,\LL\otimes \LL_0,\a\otimes\a_0}$ be the comonads on $\D(X)$ and $\D(Y)$, associated with the action of $G$ on $X$, twisted into $(\LL,\a)$, and the action of $G$ on $Y$, twisted into $(\LL\otimes \LL_0,\a\otimes\a_0)$.
Then the functor 
$\Psi_{E}$ induces a functor on descent categories $\D(X)_{\TT_{G,\LL,\a}}\ra\D(Y)_{\TT_{G,\LL\otimes \LL_0,\a\otimes\a_0}}$.
If $\Psi_{E}$ is fully faithful, then the induced functor on the descent categories is also fully faithful.

3. 
Suppose $\Psi_{E}$ is fully faithful and sends $\D^{\perf}(X)$ to $\D^{\perf}(Y)$.
Suppose the image of the functor 
$\Psi_{E}\colon\D^{\perf}(X)\ra \D^{\perf}(Y)$ is a subcategory 
$\AA\subset \D^{\perf}(Y)$.
Then
 $\Psi_{E}$ induces a fully faithful functor from 
$\D^{\perf}(X)_{\TT_{G,\LL,\a}}$ to $\D^{\perf}(Y)_{\TT_{G,\LL\otimes \LL_0,\a\otimes\a_0}}$,
whose image is the subcategory $\AA_{\TT_{G,\LL\otimes \LL_0,\a\otimes\a_0}}\subset \D^{\perf}(Y)_{\TT_{G,\LL\otimes \LL_0,\a\otimes\a_0}}$.
Also, the subcategory $\AA$ is invariant under the action of $G$.
\end{predl}
\begin{proof}
1.~Compatibility of $\Psi_{E}$ with the twisted actions of $G$ on $X$ and~$Y$ follows from Lemma~\ref{lemma_equivmorph} and Lemma~\ref{lemma_tensorE}.

2.~By the definition of a functor, compatible with twisted actions of $G$ on $X$ and $Y$, $\Psi_{E}$ can be extended to a functor  between cosimplicial categories, associated with the twisted actions of $G$ on $X$ and~$Y$. Therefore, by Lemma~\ref{lemma_psipsi}, the functor $\Psi_{E}$ is compatible with the comonads $\TT_{G,\LL,\a}$ and $\TT_{G,\LL\otimes \LL_0,\a\otimes\a_0}$. By Lemma~\ref{lemma_fffunctor},
the functor $\Psi_{E}$ induces a functor on descent categories $\D(X)_{\TT_{G,\LL,\a}}\ra\D(Y)_{\TT_{G,\LL\otimes\LL_0,\a\otimes\a_0}}$. The latter functor is fully faithful if $\Psi_{E}$ is fully faithful.

3.~The first statement follows from Lemma~\ref{lemma_fffunctor}
applied to the categories $\D(X)$, $\D(Y)$, their subcategories $\D^{\perf}(X)$,
$\D^{\perf}(Y)$, comonads $\TT_{G,\LL,\a}$ and $\TT_{G,\LL\otimes \LL_0,\a\otimes\a_0}$, 
and the subcategory $\AA\subset\D^{\perf}(Y)$.
Now we will show that the subcategory $\AA=\Psi_E(\Dp(X))$ is invariant under the action.
Consider the commutative diagram (where $p_X$ and $p_Y$ denote the projections
onto $X$ and $Y$):
$$\xymatrix{G\times X  \ar@<1mm>[d]^a \ar@<-1mm>[d]_{p_X} & G\times X \times Y\ar[r]^-{p_{13}} \ar[l]_-{p_{12}}\ar@<1mm>[d]^a \ar@<-1mm>[d]_{p_{23}} &
G\times Y  \ar@<1mm>[d]^a \ar@<-1mm>[d]_{p_Y} \\
X& X\times Y \ar[r]^{p_2} \ar[l]_{p_1} & Y.}$$ 
By Definition~\ref{def_invariantsubcat} we need to check that 
the subcategories $p_Y^*\AA$ and
$a^*\AA$ in $\D^{\perf}(G\times Y)$ coincide. The category $\AA$ is formed by the objects 
$\Psi_E(F)=p_{2*}(p_1^*F\otimes E)$, where $F\in \Dp(X)$. Hence 
the category $p_Y^*\AA$ is generated by shifts, cones and direct summands by the objects of the form $p_1^*H\otimes
p_Y^*p_{2*}(p_1^*F\otimes E)$, where $H\in\D^{\perf}(G)$,
$F\in\D^{\perf}(X)$. One has
$$
p_1^*H\otimes p_Y^*p_{2*}(p_1^*F\otimes  E)\cong 
p_1^*H\otimes p_{13*}p_{23}^*(p_1^*F\otimes  E)\cong 
p_1^*H\otimes p_{13*}(p_{12}^*p_X^*F\otimes p_1^*\LL^*_0\otimes a^* E).
$$
Due to Lemma~\ref{lemma_bcdecomp}, $a^*\D^{\perf}(X)=\Dp(G\times X)$.
Therefore, the object $p_X^*F$ can be obtained from the objects of the form
$p_1^*H'\otimes a^*F'$, where $H'\in \Dp(G)$, $F'\in \Dp(X)$, by taking shifts, cones and  direct summands. Consider the functor 
$$\Xi(-)= p_1^*H\otimes p_{13*}(p_{12}^*(-)\otimes p_1^*\LL^*_0\otimes a^* E)
\colon \Dp(G\times X)\to \Dp(G\times Y),$$
it is exact and preserves direct summands. So, the object
$$p_1^*H\otimes
p_Y^*p_{2*}(p_1^*F\otimes E)\cong \Xi(p_X^*F)$$
can be obtained from the objects of the form
$$\Xi(p_1^*H'\otimes a^*F')$$
by taking shifts, cones and direct summands. 
Further, we have 
\begin{multline*}
\Xi(p_1^*H'\otimes a^*F')=
p_1^*H\otimes p_{13*}(p_{12}^*(p_1^*H'\otimes a^*F')\otimes p_1^*\LL^*_0\otimes a^* E)\cong\\
\cong
p_1^*H\otimes p_{13*}(p_{12}^*a^*F'\otimes p_1^*(\LL^*_0\otimes H')\otimes a^* E)\cong
p_1^*H\otimes p_{13*}(a^*(p_{1}^*F'\otimes  E)\otimes
p_{13}^*p_1^*(\LL^*_0\otimes H'))\cong \\
\cong p_1^*H\otimes p_1^*(\LL^*_0\otimes H') \otimes
p_{13*}a^*(p_{1}^*F'\otimes  E)\cong p_1^*(H\otimes H'\otimes \LL_0^*)
\otimes a^*p_{2*}(p_{1}^*F'\otimes  E).
\end{multline*}
Since $p_{2*}(p_{1}^*F'\otimes  E)=\Psi_E(F')\in\AA$, the object $\Xi(p_1^*H'\otimes a^*F')$ belongs to $a^*\AA$.
The subcategory $a^*\AA$ is triangulated and closed under direct summands, hence 
all objects $p_1^*H\otimes
p_Y^*p_{2*}(p_1^*F\otimes E)$ for $H\in \Dp(G), F\in \Dp(X)$ lie in $a^*\AA$. 
Thus, $p_Y^*\AA\subset a^*\AA$. 
The opposite inclusion is checked in the same way.
\end{proof}

\section{Semiorthogonal decompositions for  varieties with an invariant exceptional collection}
\label{section_sodforexcoll}

In this section we describe components of the semiorthogonal decomposition from Theorem~\ref{th_sodforequiv} in the case when the invariant semiorthogonal decomposition of the derived category of coherent sheaves on $X$ is
generated by an exceptional collection.

Let $X$ be a quasi-projective scheme over a field $\k$.
By a group $G$ in this section we understand a reduced affine group scheme over $\k$.

Suppose that $E$ is an exceptional object in the category $\D^{\perf}(X)$. It generates the subcategory $\langle E\rangle\subset\D^{\perf}(X)$.
We will show that invariance of this subcategory under an action of $G$ on $X$ (see Definition~\ref{def_invariantsubcat}) is related with invariance of the object $E$ in the sense, introduced below.

For an action of a finite group $G$ on $X$ it is natural to say that a sheaf~$F$ on $X$ is preserved by the action if $g^*F\cong F$
for all $g\in G$. This definition does not work well for algebraic groups because the group may have too little rational points. To deal with invariant objects we introduce the following
\begin{definition}
\label{def_preserves}
An action of an algebraic group $G$ on a scheme $X$ \emph{preserves } an object of the derived category $F\in\D(X)$ if there is a line bundle
$\LL$ on $G$ such that the objects $p_1^*\LL\otimes p_2^*F$ and $a^*F$ on $G\times X$ are quasi-isomorphic.
\end{definition}

\begin{predl}
\label{prop_invariantsheaf2} For an exceptional object $E\in\D^{\perf}(X)$ on a projective scheme $X$ over a field $\k$, acted by a reduced affine group scheme~$G$ of finite type over $\k$, the below conditions are equivalent:
\begin{enumerate}
\item for some line bundle $\LL$ on $G$ there exists an isomorphism $p_1^*\LL\otimes p_2^*E\cong a^*E$ in the category $\D^{\perf}(G\times X)$;
\item for any closed point $g$ of the scheme $G$ with the residue field $\k(g)$ one has
$g^*E'\cong E'$, where the object $E'$ is obtained from $E$ by the extension of scalars $\k\ra \k(g)$;
\item the subcategories $p_2^*\langle E\rangle$ and $a^*\langle E\rangle$
in $\D^{\perf}(G\times X)$ coincide.
\end{enumerate}
\end{predl}
\begin{proof}
2 $\Rightarrow$ 1 is proved in~\cite[prop. 2.17]{El_C1}.

1 $\Rightarrow$ 3. By the definition, the subcategory $p_2^*\langle E\rangle$ in $\D^{\perf}(G\times X)$ is generated by the objects of the form $p_1^*F\otimes p_2^*F'$,
where $F\in \D^{\perf}(G)$, $F'\in\langle E\rangle$. Particularly, $p_2^*\langle E\rangle$ contains the object $a^*E\cong p_1^*\LL\otimes p_2^*E$. Therefore,
$p_2^*\langle E\rangle\supset a^*\langle E\rangle $. The inverse inclusion is proved in the same way.

3 $\Rightarrow$ 2. Let $g$ be a closed point of the scheme $G$.
Restrict two coinciding subcategories $p_2^*\langle E\rangle$ and $a^*\langle E\rangle $
in $\D^{\perf}(G\times X)$ on the fiber $X'=g\times X$. We get two subcategories in  $\D^{\perf}(X')$ generated by the exceptional objects $E'$ and  $g^*E'$ respectively. We conclude that $E'$ and
$g^*E'$ are isomorphic.
\end{proof}

The following fact is crucial for us: any exceptional object $E$ on $X$ preserved by the group action possesses an equivariant structure for some cocycle on $G$.
\begin{predl}
\label{prop_invar>equiv}
Let $E$ be an exceptional object in $\D(X)$ preserved by the action of  a group $G$ on $X$.
Then for some cocycle $(\LL,\a)$ on~$G$ there is an
$(\LL,\a)$-equivariant structure on $E$. That is, there is an isomorphism
$\theta\colon p_1^*\LL\otimes p_2^*E\ra a^*E$ compatible with $\a$
in the sense of Definition~\ref{def_Lasheaf}.
\end{predl}
\begin{proof}
Let us fix an isomorphism
$\theta\colon p_1^*\LL\otimes  p_2^*E\ra a^*E$. On the product
$G\times G\times X$  we have a commutative diagram
$$\xymatrix{
p_1^*\LL\otimes p_2^*\LL\otimes p_3^*E \ar[rr]^{1\otimes p_{23}^*\theta}
\ar[d]_{\a'}&&
p_1^*\LL\otimes (ap_{23})^*E \ar[rr]^{(1\times a)^*\theta} &&
(a(1\times a))^*E  \ar@{=}[d]\\
(\mu p_{12})^*\LL\otimes p_3^*E \ar[rrrr]^{(\mu \times 1)^*\theta} &&&&
(a(\mu \times 1))^*E,
}
$$
where
$$\a'\colon
p_1^*\LL\otimes p_2^*\LL\otimes p_3^*E\ra
(\mu p_{12})^*\LL\otimes p_3^*E
$$
is a certain isomorphism.
Let us check that $\a'$
have the form $p_{12}^*\a\otimes 1$ for some isomorphism $\a\colon  p_1^*\LL\otimes p_2^*\LL\ra\mu^*\LL$ on $G\times G$.
Indeed, since $E$ is exceptional,
\begin{multline*}
\Hom(p_1^*\LL\otimes p_2^*\LL\otimes p_3^*E,
(\mu p_{12})^*\LL\otimes p_3^*E)=\\
=\Hom(p_3^*E,
(p_1^*\LL\otimes p_2^*\LL)^{-1}\otimes(\mu p_{12})^*\LL\otimes p_3^*E)=\\
=\Hom(E,
p_{3*}((p_1^*\LL\otimes p_2^*\LL)^{-1}\otimes(\mu p_{12})^*\LL\otimes p_3^*E))=\\
=\Hom(E,
p_{3*}((p_1^*\LL\otimes p_2^*\LL)^{-1}\otimes(\mu p_{12})^*\LL)\otimes E)=\\
=\Hom(E,H^0(G\times G,(p_1^*\LL\otimes p_2^*\LL)^{-1}\otimes
\mu^*\LL)\otimes E)=
\Hom(p_1^*\LL\otimes p_2^*\LL,\mu^*\LL).
\end{multline*}
The last but one equality here is due to the flat base change theorem.
Considering sheaves and their morphisms on $G\times G\times G\times X$,
we get the associativity condition: on $G\times G\times G$
the isomorphisms $$(1 \times \mu)^*\a\circ (1\otimes p_{23}^*\a) \quad\text{and}\quad
(\mu \times 1)^*\a\circ (p_{12}^*\a\otimes 1)
$$
between the sheaves
$p_1^*\LL\otimes p_2^*\LL\otimes p_3^*\LL$ and $(\mu(1\times \mu))^*\LL$
are equal. Hence, the pair $(\LL,\a)$ is a cocycle on the group $G$
in the sense of Definition~\ref{def_la} and $\EE=(E,\theta)$ is an equivariant object.
\end{proof}

Suppose the category of perfect complexes $\D^{\perf}(X)$ possesses a full exceptional collection $(E_1,\ldots,E_n)$. This collection induces a semiorthogonal decomposition $$\D^{\perf}(X)=\left\langle\langle E_1\rangle,\ldots,\langle E_n\rangle\right\rangle$$
into the categories, equivalent to  $\D^b(\k\mmod)$. Suppose that the objects~$E_i$ are invariant under the action of $G$.
Proposition~\ref{prop_invariantsheaf2} implies that the components of the above decomposition are preserved by the action.
By Proposition~\ref{prop_invar>equiv}, the object $E_i$ can be endowed with an equivariant structure for a certain cocycle $(\LL_i,\a_i)$ on~$G$. Denote the corresponding twisted equivariant object of the category $\D^{\perf}(X)^{G,\LL_i,\a_i}$ by $\EE_i$. Suppose that the group $G$ is linearly reductive. Then by theorem~\ref{th_descentforequivtwist}, $\EE_i$ corresponds to an object of the derived  category $\D^{\perf,G,\LL_i,\a_i}(X)$, which will also by denoted by~$\EE_i$.

\begin{theorem}
\label{th_sodforexcoll}
The category $\D^{\perf,G}(X)$ admits a semiorthogonal decomposition $$\D^{\perf,G}(X)=\langle\EE_1\otimes\D^b(\repr(G,\LL_1^{-1},\a_1^{-1})),\ldots,
\EE_n\otimes\D^b(\repr(G,\LL_n^{-1},\a_n^{-1}))\rangle.$$
\end{theorem}
\begin{remark}
This fact was proved in~\cite[th. 2.11]{El_C1}
under different assumptions: the objects of the exceptional collection were supposed to be sheaves (this allowed to avoid using descent theory for derived categories) and the group was not assumed to be linearly reductive.
\end{remark}
\begin{proof}
Let $\TT_G$ be the comonad on the category $\D(X)$ associated with the morphism of stacks $X\ra X\quot G$.
By Theorem~\ref{th_descentfor2}, the category $\D^{\perf,G}(X)$ is equivalent to the descent category $\D^{\perf}(X)_{\TT_{G}}$.
Theorems~\ref{th_sodforequiv} imply that $\D^{\perf}(X)_{\TT_{G}}$ has the semiorthogonal decomposition $\langle\langle E_1\rangle_{\TT_{G}},\ldots,\langle E_n\rangle_{\TT_{G}}\rangle$.
Here $\langle E_i\rangle_{\TT_{G}}$ is the descent category, related to the subcategory $\langle E_i\rangle\subset \D^{\perf}(X)$, see Definition~\ref{def_dxt'}.
Let us describe the categories $\langle E_i\rangle_{\TT_{G}}$ explicitly using Lemma~\ref{lemma_fffunctor}.

Consider the functor $\Psi_i= - \otimes E_i\colon \D^b(\k\mmod)=\D^{\perf}(\Spec\k)\ra \D^{\perf}(X)$. Its image is the subcategory $\langle E_i\rangle\subset\D^{\perf}(X)$. This functor extends to a fully faithful functor $\D(\k\Mod)\cong\D(\Spec\k)\ra \D(X)$, which also will be denoted by $\Psi_i$.
Clearly, $\Psi_i$ is a kernel functor, it can be defined by the kernel $E_i\in\D^{\perf}(\Spec \k\times X)$. This kernel admits a structure of an equivariant object twisted into the cocycle $(\LL_i,\a_i)$. By Proposition~\ref{prop_kernel}, the functor $\Psi_i$
induces a fully faithful functor $$\D^b(\repr(G,\LL_i^{-1},\a_i^{-1}))\cong\D^{\perf}(\Spec\k)_{\TT_{G,\LL_i^{-1},\a_i^{-1}}}
\longrightarrow \D^{\perf}(X)_{\TT_G}\cong\D^{\perf,G}(X),$$
its image is the subcategory $\langle E_i\rangle_{\TT_G}$.
\end{proof}

\section{Semiorthogonal decompositions for projective bundles and for blow-ups}\label{section_sodforprojbund}

In this section we apply Theorem~\ref{th_sodforequiv} in two special cases of group actions on a scheme and describe explicitly the components of the resulting semiorthogonal decompositions.

The first example is the equivariant derived category of a projective bundle, studied in~\cite{El_C3} in the case of finite groups. The semiorthogonal decomposition of the derived category of projective bundles, from which we start, was constructed by D.\,Orlov in~\cite{Or_projbundles}. Here we consider the simple case when the group action on the projective bundle is induces by an equivariant structure on the base.

Let $S$ be a quasi-projective scheme over a field $\k$, let $E$ be a vector bundle of rank $r$ on $S$, let $X=\P_S(E)$ be its projectivization. Denote by $\pi\colon X\ra S$ the natural projection.
Then the semiorthogonal decomposition by Orlov has the following form:
\begin{equation}
\label{equation_decompproj}
\D^{\perf}(X)
=\langle \pi^*\D^{\perf}(S),
\O_{X/S}(1)\otimes \pi^*\D^{\perf}(S),\ldots, \O_{X/S}(r-1) \otimes
\pi^*\D^{\perf}(S)\rangle.
\end{equation}
Suppose that a group $G$ (more precisely, an affine group scheme of finite type over $\k$) acts on $S$ and that $E$ admits a $G$-equivariant structure twisted into the cocycle $(\LL,\a)$. We denote the corresponding $(\LL,\a)$-$G$-equivariant bundle by $\EE$. Then the action of the group $G$ on~$X$ is well-defined, and the projection $\pi$ is an equivariant map.
Suppose that the group~$G$ is linearly reductive.
\begin{theorem}
\label{th_sodforprojbund}
In the above assumptions the following semiorthogonal decomposition exists
\begin{multline*}
\D^{\perf,G}(X)=\langle \pi^*\D^{\perf,G}(S), \O_{X/S}(1)\otimes \pi^*\D^{\perf,G,\LL,\a}(S),\ldots\\
\ldots,\O_{X/S}(r-1)\otimes \pi^*\D^{\perf,G,\LL^{r-1},\a^{r-1}}(S)\rangle.
\end{multline*}
\end{theorem}
\begin{proof}
First, note that $\pi^*\EE$ is an $(\LL,\a)$-equivariant bundle on $X$,
and  $\O_{X/S}(-1)$ is its $1$-dimensional $(\LL,\a)$-equivariant subbundle. Therefore, the bundles $\O_{X/S}(k)$ are twisted
$(\LL,\a)^{-k}$-$G$-equivariant subbundles. For any integer $k, 0\le k\le r-1$, consider the fully faithful functor
$$\Psi_k=\O_{X/S}(k)\otimes \pi^*(-)\colon \D(S)\ra \D(X).$$
This functor is kernel, the corresponding kernel is the sheaf
$\O_{X/S}(k)$,  located on the graph of the map $\pi\colon X\ra
S$. Then, $\Psi_k$ sends perfect complexes to perfect complexes, the category $\Psi_k(\D^{\perf}(S))$ is the subcategory
$\O_{X/S}(k)\otimes \pi^*\D^{\perf}(S)$
in the decomposition~(\ref{equation_decompproj}). Let $\TT_{G,\LL^k,\a^k}$ and
$\TT_G$ be the comonads on the categories $\D(S)$ and $\D(X)$, associated with the actions of $G$ on $S$ (twisted into the cocycle $(\LL^k,\a^k)$)
and on $X$
(the standard one). By Proposition~\ref{prop_kernel}, the functor
$\Psi_k$ induces a fully faithful functor
$$\D^{\perf,G,\LL^k,\a^k}(S)\cong\D^{\perf}(S)_{\TT_{G,\LL^k,\a^k}}\ra
\D^{\perf}(X)_{\TT_G}\cong\D^{\perf,G}(X),$$
whose image is the descent category $(\O_{X/S}(k)\otimes \pi^*\D^{\perf}(S))_{\TT_G}$.
To complete the proof, note that the subcategories $\O_{X/S}(k)\otimes \pi^*\D^{\perf}(S)$
of the decomposition~(\ref{equation_decompproj}) are invariant under the group action, and apply Theorem~\ref{th_sodforequiv} to the decomposition~(\ref{equation_decompproj}).
\end{proof}

The second example is a blow-up of a smooth invariant subvariety in a smooth variety.

Let $Z\subset X$ be a smooth subscheme of codimension $r$ of a smooth quasi-projective scheme over a field $\k$, let $\s\colon \overline X\ra X$ be the blow-up of $X$ along $Z$. Denote by $\overline Z$ the preimage of $Z$ under $\s$, then $\overline Z$ is the projectivization of the normal bundle $\NN_{Z/X}$ on~$Z$. Denote by $\s_Z$ the restriction of $\s$ to $\overline Z$. As in the previous example, the starting semiorthogonal decomposition for the blow-up was constructed by Orlov in~\cite{Or_projbundles}:
\begin{equation}
\label{equation_decompblowup}
\D^{\perf}(\overline X)=\langle \O_{\overline Z/Z}(-r+1)\otimes \s_Z^*\D^{\perf}(Z), \ldots, \O_{\overline Z/Z}(-1)\otimes \s_Z^*\D^{\perf}(Z), \s^*\D^{\perf}(X) \rangle.
\end{equation}

Let an affine linearly reductive group scheme $G$ of finite type over $\k$  act on the scheme $X$ such that the subscheme $Z$ is invariant.
\begin{theorem}
\label{th_blowups}
In the above assumptions there exists the following semiorthogonal decomposition
$$
\D^{\perf,G}(\overline X)=\langle \O_{\overline Z/Z}(-r+1)\otimes \s_Z^*\D^{\perf,G}(Z),
\ldots, \O_{\overline Z/Z}(-1)\otimes \s_Z^*\D^{\perf,G}(Z), \s^*\D^{\perf,G}(X) \rangle,
$$
whose components $\O_{\overline Z/Z}(-i)\otimes \s_Z^*\D^{\perf,G}(Z)$
are equivalent to the category $\D^{\perf,G}(Z)$,  and the component $\s^*\D^{\perf,G}(X)$ -- to the category $\D^{\perf,G}(X)$.
\end{theorem}
\begin{proof}
Let us sketch the proof, details are analogous to the proof of Theorem~\ref{th_sodforprojbund}.

The component $\s^*\D^{\perf}(X)$ of  decomposition~(\ref{equation_decompblowup}) is the image
of the fully faithful functor $$\s^*\colon \D^{\perf}(X)\ra \D^{\perf}(\overline X).$$ This functor is compatible with the actions of $G$ on $X$ and on $\overline X$, therefore the subcategory $\s^*\D^{\perf}(X)\subset \D^{\perf}(\overline X)$ is preserved by the action and the corresponding component of the semiorthogonal decomposition of
$\D^{\perf,G}(\overline X)$ is equivalent to $\D^{\perf,G}(X)$.

Similarly, the component $\O_{\overline Z/Z}(-i)\otimes \s_Z^*\D^{\perf}(Z)$
of  decomposition~(\ref{equation_decompblowup})
is the image of the fully faithful functor
$$\Psi_i=j_*(\O_{\overline Z/Z}(-i)\otimes \s_Z^*(-))\colon\D^{\perf}(Z)\ra \D^{\perf}(\overline X),$$
where $j$ denoted the embedding of $\overline Z$ into $\overline X$.
The normal bundle $\mathcal N_{Z/X}$  on~$Z$ is equivariant,
hence the bundles $\O_{\overline Z/Z}(i)$ on $\overline Z$ are also equivariant.
The functors~$\Psi_i$ are compatible with the actions of $G$ on $Z$ and on  $\overline X$, hence the components $\O_{\overline Z/Z}(-i)\otimes \s_Z^*\D^{\perf}(Z)$ of the decomposition~(\ref{equation_decompblowup}) are preserved by the action and the components $\O_{\overline Z/Z}(-i)\otimes \s_Z^*\D^{\perf,G}(Z)$
are equivalent to the category $\D^{\perf,G}(Z)$.
\end{proof}

\end{document}